\documentclass[11pt]{article}
    \usepackage{url}
    \usepackage{verbatim}
    \usepackage[titletoc]{appendix}
    \usepackage{graphicx}
	\usepackage{amsmath}
	\usepackage{amsthm}
	\usepackage{amsfonts}
	\usepackage{graphicx}
    \usepackage{nicefrac}
    \usepackage{longtable}
    \usepackage{color}

    \newcommand{\norm}[2]{\left\| #1 \right\|_{#2}}
    \newcommand{\iprd}[3]{\left( #1 , #2 \right)_{#3}}
    \newcommand{\hf}{\frac{1}{2}}
    \newcommand{\msfJ}{\mathsf J}

    \newcommand{\Dh}{\Delta_h}

    \newcommand{\nrm}[1]{\left\| #1 \right\|}
    \newcommand{\eip}[2]{\left[ #1 \middle| #2 \right]}
    \newcommand{\cip}[2]{\left( #1 \middle| #2 \right)}
    \newcommand{\eipns}[2]{\left[ #1 \middle\| #2 \right]_{\rm ns}}
    \newcommand{\eipew}[2]{\left[ #1 \middle\| #2 \right]_{\rm ew}}
    
    \newcommand{\ciptwo}[2]{\left( #1 \middle\| #2 \right)}
    \newcommand{\cipthree}[2]{\left( #1 \middle| \! \middle| \! \middle| #2 \right)}
    \newcommand{\convone}[2]{\left[ #1 \star #2 \right]}
    \newcommand{\convtwo}[2]{\left[ #1 \star #2 \right]}

    \def\g{\mbox{\boldmath $g$}}
    \def\h{\mbox{\boldmath $h$}}
    
	\newtheorem{thm}{Theorem}[section]
	\newtheorem{prop}[thm]{Proposition}
	\newtheorem{cor}[thm]{Corollary}
	
	\newtheorem{lem}[thm]{Lemma}
	\newtheorem{rem}[thm]{Remark}

\begin{document}
	\title{Convergence Analysis for Second Order Accurate Convex Splitting Schemes for the Periodic Nonlocal Allen-Cahn and Cahn-Hilliard Equations}
	
	\author{
Zhen Guan\thanks{Department of Mathematics, The University of California, Irvine, CA 92697, USA  (zguan2@math.uci.edu)}
	\and
John Lowengrub\thanks{Department of Mathematics, The University of California, Irvine, CA 92697, USA (lowengrb@math.uci.edu)}
	\and
Cheng Wang \thanks{Department of Mathematics, The University of Massachusetts, Dartmouth, MA 02747, USA (Corresponding Author: cwang1@umassd.edu)} }

	\maketitle
	\numberwithin{equation}{section}

	\begin{abstract}
In this paper we provide a detailed convergence analysis for fully discrete second order (in both time and space) numerical schemes for nonlocal Allen-Cahn (nAC) and nonlocal Cahn-Hilliard (nCH) equations. The unconditional unique solvability and energy stability ensures $\ell^4$ stability. The convergence analysis for the nAC equation follows the standard procedure of consistency and stability estimate for the numerical error function. 
For the nCH equation, due to the complicated form of the nonlinear term, a careful expansion of its discrete gradient is undertaken and an $H^{-1}$ inner product estimate of this nonlinear numerical error is derived to establish convergence. 
In addition, an a-priori $W^{1,\infty}$ bound of the numerical solution at 
the discrete level is needed in the error estimate. Such a bound can be 
obtained by performing a higher order consistency analysis by using asymptotic expansions for the numerical solution. Following the technique originally proposed by Strang (e.g., 1964), instead of the standard comparison between the exact and numerical solutions, an error estimate between the numerical solution and the constructed approximate solution yields an $O( s^3 + h^4)$ convergence in $\ell^\infty (0, T; \ell^2)$ norm, which leads to the necessary bound under a standard constraint $s \le C h$. Here, we also prove convergence of the scheme in the maximum norm under the same constraint. 
	\end{abstract}

	\section{Introduction}
	
	In this paper our primary goal is to develop a detailed convergence analysis of fully discrete second order (in both time and space) numerical schemes for integro-partial-differential equations:
	\begin{equation}
     \partial_t \phi = - M w ,  \, \, \, M \ge 0 ,   
     \label{pde-non conserve}
     \end{equation}
and 
   \begin{equation}
      \partial_t \phi = \Delta w, \label{pde-conserve}
	\end{equation}
where
	\begin{equation}
	\label{chemical_potential}
w := \delta_{\phi} E = \phi^3 +\gamma_c \phi -\gamma_e \phi + (\msfJ*1)\phi - \msfJ*\phi ,
	\end{equation}
and $\msfJ$ is a convolution kernel. These equations are defined in $(0,T] \times \Omega$, with the initial condition
\begin{equation}
\phi(0, x ) = \phi_0 (x) ,
\end{equation}
where $\Omega$ is a rectangular domain in $\mathbb{R}^n$. The solution $\phi(x,t)$ is $\Omega$ -- periodic in space. Suppose  $\Omega = (0,L_1)\times(0,L_2) \subset\mathbb{R}^2$. Define $C^m_{\rm per}(\Omega) = \left\{ f \in C^m\left(\mathbb{R}^2\right)  \ \middle| \ f \mbox{ is $\Omega$ -- periodic}\right\}$, where $m$ is a positive integer.  Set $C^\infty_{\rm per}(\Omega) = \bigcap_{m=1}^\infty C^m_{\rm per}(\Omega)$. For any $1\le q<\infty$, define $L^q_{\rm per}(\Omega) = \left\{ f \in L_{\rm loc}^q\left(\mathbb{R}^2\right)  \ \middle| \ f \mbox{ is $\Omega$ -- periodic}\right\}$. Of course, $L^q(\Omega)$ and $L_{\rm per}^q(\Omega)$ can be identified in a natural way.  Suppose that the convolution kernel $\msfJ:\mathbb{R}^2 \to \mathbb{R}$ satisfies
	\begin{description}
	\label{condition-J}
 	\item
A1. $\msfJ = \msfJ_c-\msfJ_e$, where $\msfJ_c,\, \msfJ_e \in C_{\rm per}^\infty\left(\Omega\right)$ are non-negative.
	\item
A2. $\msfJ_c$ and $\msfJ_e$ are even, \emph{i.e.}, $\msfJ_\alpha(x_1,x_2)=\msfJ_\alpha(-x_1,-x_2)$, for all $x_1,x_2\in \mathbb{R}$,  $\alpha = c$, $e$.
	\item
A3. $\int_\Omega \msfJ({\bf x})\, d{\bf x} > 0$.
	\end{description}
Given $\phi \in L_{\rm per}^2(\Omega)$, by $\msfJ*\phi$ we mean the circular, or periodic, convolution defined as
	\begin{equation}
\left(\msfJ*\phi\right)({\bf x}) := \int_\Omega \msfJ({\bf y}) \, \phi({\bf x}-{\bf y}) \, d{\bf y} = \int_\Omega \msfJ({\bf x}-{\bf y}) \, \phi({\bf y}) \, d{\bf y} .
	\end{equation}
Clearly $\msfJ*1=\int_\Omega \msfJ({\bf x})\, d{\bf x}$ is a positive constant. We refer to Eqs.~\eqref{pde-non conserve} and~\eqref{pde-conserve} as the nonlocal Allen-Cahn (nAC) and nonlocal Cahn-Hilliard (nCH) equations, respectively. For any $\phi\in L_{\rm per}^4(\Omega)$, we define
	\begin{equation}
E(\phi) = \frac{1}{4}\norm{\phi}{L^4}^4 + \frac{\gamma_c-\gamma_e+\msfJ*1}{2}\norm{\phi}{L^2}^2 -\frac{1}{2}\iprd{\phi}{\msfJ*\phi}{L^2} ,
	\label{non-local-energy}
	\end{equation}
where $\gamma_c,\, \gamma_e \ge 0$ are constants.  Typically, one of $\gamma_c$ and $\gamma_e$ is zero.   The energy (\ref{non-local-energy}) is equivalent to 
	\begin{equation}\label{energy}
E(\phi) = \int_\Omega\left\{\frac{1}{4}\phi^4+ \frac{\gamma_c-\gamma_e}{2}\phi^2 + \frac{1}{4}\int_\Omega \msfJ({\bf x}-{\bf y})\left(\phi({\bf x})-\phi({\bf y})\right)^2d{\bf y} \right\}d{\bf x} .
	\end{equation} 
The term $w$ defined in \eqref{chemical_potential} is the chemical potential (variational derivative) relative to the energy (\ref{energy}). Formally, periodic solutions of the nAC and nCH equations, Eq.~(\ref{pde-non conserve}) and Eq.~(\ref{pde-conserve}), dissipate $E$ at the rate $d_t E(\phi) = - M \| w \|_{L^2}^2$, $d_t E(\phi) = -\norm{\nabla w}{L^2}^2$, respectively.  Eqs.~\eqref{pde-non conserve} and~\eqref{pde-conserve} can also be rewritten as
	\begin{eqnarray} 
    &&  
\partial_t \phi = - M ( a(\phi) \phi  - \msfJ*\phi ) ,  
  \quad  \mbox{(nAC)} ,	
\\
	  &&  
\partial_t \phi = \nabla\cdot\left(a(\phi)\nabla \phi \right)  -\left(\Delta \msfJ\right)*\phi ,  
  \quad  \mbox{(nCH)} ,
	\end{eqnarray}
where
	\begin{equation}
a(\phi) = 3\phi^2 +\gamma_c-\gamma_e + \msfJ*1  .
	\end{equation}
We refer to $a(\phi)$ as the diffusive mobility, or just the diffusivity.  To make Eq.~(\ref{pde-conserve}) positive diffusive (and non-degenerate), we require~\cite{bates05a}
	\begin{equation}
\gamma_c-\gamma_e +\msfJ*1 =: \gamma_0 > 0  ,
	\label{gamma-0}
	\end{equation}
in which $a(\phi) >0$.  We will assume that (\ref{gamma-0}) holds in the sequel. 

Eq.~(\ref{pde-conserve}) can also be viewed as a special case of a more general form of nonlocal Cahn-Hilliard equation which is defined as the following:

	\begin{equation}
	\label{pde-conserve-intro}
\partial_t \phi = \nabla\cdot \left(M(\phi)\nabla w_{\text{general}} \right) ,
	\end{equation}
where $M(\phi)>0$ is the mobility, $w_{\text{general}}=\delta_{\phi}E_{\text{general}}$ is the chemical potential and $E_{\text{general}}$ is the nonlocal interaction energy defined as
\begin{equation}\label{energy-intro}
 E_{\text{general}}=\int_{\Omega} \left(F(\phi) - \frac{1}{2}\int_{\Omega} \msfJ({\bf x}-{\bf y})\phi({\bf x})\phi({\bf y}) d{\bf y}\right) d{\bf x} .
\end{equation}
Here we take $\phi \in C^\infty_{\rm per}(\Omega)$ in the sequel, and $F$ is the local energy density, which is usually nonlinear. One important example of  the energy density $F$ is 
\begin{equation}
F(\phi)=\frac{1}{4}\phi^4+\frac{\tilde{\gamma}_c-\tilde{\gamma}_e}{2}\phi^2 ,
\end{equation}
where $\tilde{\gamma}_c,\, \tilde{\gamma}_e \ge 0$ are constants.
By rewriting the nonlocal part of $E_{\text{general}}$ as
\begin{equation}
- \frac{1}{2} \int_{\Omega} \int_{\Omega} \msfJ({\bf x}-{\bf y})\phi({\bf x})\phi({\bf y}) d{\bf y} d{\bf x} = \frac{1}{4}\int_\Omega\int_\Omega \msfJ({\bf x}-{\bf y})\left(\phi({\bf x})-\phi({\bf y})\right)^2d{\bf y} d{\bf x} - \int_\Omega\frac{\msfJ*1}{2}\phi({\bf x})d{\bf y} , 
\end{equation}
and
\begin{equation}
\tilde{\gamma}_c-\tilde{\gamma}_e = \gamma_c-\gamma_e + \msfJ*1, 
\end{equation}
we can obtain Eq.~\eqref{pde-conserve}. Another important example is a regular solution model with free energy
\begin{equation}\label{entropy-eq}
 F(\phi) := \theta [\phi \log (\phi) + (1-\phi) \log (1-\phi)] - 2\theta_c \phi(1-\phi),
\end{equation}
where $\theta$ and $\theta_c$ represent the absolute and critical temperatures, respectively~\cite{barrett99, elder02, greenwood10, kim04}.

Eq.~\eqref{pde-conserve-intro} is the nonlocal version of the classical Cahn-Hilliard (CH) equatios, for which the free energy is~\cite{allen79, cahn61, cahn58}:
\begin{equation}
	\label{eq:G-L-potential}
E_{\text{local}}(\phi) = \int_{\Omega}\left(G(\phi) + \frac{\epsilon^2}{2}|\nabla \phi|^2\right) d{\bf x} , 
	\end{equation}
and the dynamical equation becomes 
	\begin{equation}
	\label{eq:classical-CH}
\partial_t\phi =  \nabla \cdot \left(M(\phi) \nabla \delta_{\phi}  E_{\text{local}} \right) = \nabla \cdot \left(M(\phi) \nabla \left(G'(\phi) - \epsilon^2 \Delta \phi\right) \right) .
	\end{equation}
The connection between two energies can be seen as follows. The energy $E_{\text{general}}$ can be written as
	\begin{equation}
E_{\text{general}}(\phi) = \iprd{F(\phi)}{1}{L^2} -\frac{1}{2}\iprd{\phi}{\msfJ*\phi}{L^2} .
	\label{non-local-general}
	\end{equation}
The term $E_{\text{local}}(\phi)$ is obtained by approximating the nonlocal energy $E_{\text{general}}$~\cite{archer04a, bates06b, evans79, horntrop01, wise08}. Specifically, one takes the approximation $\left(\msfJ*\phi\right) \approx \msfJ_0 \phi + \frac{1}{2}\msfJ_2 \Delta \phi$ where $\msfJ_0 = \int_{\Omega} \msfJ({\bf x}) d{\bf x}$ and $\msfJ_2 = \int_{\Omega} \msfJ({\bf x})|{\bf x}|^2 \, d{\bf x}$ ($\msfJ_2$ is the second moment of $\msfJ$). Under the assumption of periodic boundary conditions, we have the following approximation: 
	\begin{equation}
\frac{1}{2}\iprd{\phi}{\msfJ*\phi}{L^2} \approx \frac{1}{2}\iprd{\phi}{\msfJ_0 \phi + \frac{1}{2}J_2 \Delta \phi}{L^2} = \frac{\msfJ_0}{2}\iprd{\phi^2}{{\bf 1}}{L^2} + \frac{\msfJ_2}{4} \int_{\Omega} |\nabla \phi |^2 d{\bf x} \ .
	\end{equation} 
Thus one can obtain $E_{\text{local}}$ with $G(\phi)=F(\phi)-\frac{\msfJ_0}{2}\phi^2$ and $\epsilon^2 = \frac{\msfJ_2}{2}$. 

General nCH and nAC equations have been widely used in many fields ranging from physics and material science to biology, finance and image processing. In materials science, Eqs.~(\ref{pde-conserve-intro}), (\ref{eq:classical-CH}) and other closely related equations arise as mesoscopic models of interacting particle systems~\cite{archer04b, horntrop01, demasi94b}. Eq.~\eqref{pde-non conserve} is used to model phase transition~\cite{bates97}. In Dynamic Density Functional Theory (DDFT)~\cite{archer04a, archer04b}, the interaction kernel $\msfJ=c^{(2)}(x,y|\phi_{ref})$ is the two-particle direct correlation function, $\phi$ represents the mesoscopic particle density and $\phi_{ref}$ is the average density. See~\cite{evans79, wise10b, likos07, marconi99, rogers91, rogers96, wise08} for further details. In biology, $\msfJ$ has been used to model interactions among cells and extracellular matrix~\cite{armstrong06, armstrong09, chauviere12, gerisch08}. In mathematical models of finance, $\msfJ$ arises from an expectation taken with respect to a particular measure which is used in the model for option pricing~\cite{merton76}. In the modeling for image segmentation with nCH equation, $\msfJ$ is interpreted as the attracting force~\cite{gajewski05, gajewski03}. Readers are referred to~\cite{bates06a, bates05b, bates05a, caginalp11, gajewski03, giacomin97, giacomin98} for theoretical studies of general nCH equations , and~\cite{bates97, bates06b, fife03} for general nAC equations.

There are a few works dedicated to numerical simulation of, or numerical methods for, equations like (\ref{pde-non conserve}) and (\ref{pde-conserve}).   Anitescu~\emph{et al.}~\cite{anitescu04} considered an implicit-explicit time stepping framework for a nonlocal system modeling turbulence, where, as here, the nonlocal term is treated explicitly.  References~\cite{du12a,zhou10} address finite element approximations (in space) of nonlocal peridynamic  equations with various boundary conditions.  In~\cite{bates09}, a  finite difference  method for Eq.~(\ref{pde-non conserve}) with non-periodic boundary conditions is applied and analyzed.   Reference~\cite{hartley09} uses a spectral-Galerkin method to solve a nonlocal Allen-Cahn type problem, like Eq.~\eqref{pde-non conserve}, but with a stochastic noise term and equation modeling heat flow.  For other references for equations like (\ref{pde-conserve}), see~\cite{abukhdeir11, gajewski05, horntrop01, sachs08}. A first order convex splitting for Eq.~\eqref{pde-conserve} scheme was introduced and analyzed in \cite{guan12a}. 
A second order convex splitting scheme for the general system Eq.~\eqref{pde-conserve-intro} was introduced in \cite{guan12b};  
its unconditional energy stability and unique solvability were presented. 

Here, we present a detailed convergence analysis of fully discrete second order convex splitting schemes using the specific forms of Eqs.~\eqref{pde-non conserve} and \eqref{pde-conserve}. We prove convergence in both the $\ell^2$ and $\ell^\infty$ spatial norms. We note that this convergence analysis is much more challenging 
than that of the first order scheme, primarily due to its complicated form for the 
nonlinear term and lack of higher order diffusion term. The present work focuses on the 2D case but can be straightforwardly extended to 3D. Although the result can not be easily extended to more general nonlinear local densities $F$, our results are still useful because polynomial local density functions are widely used. 

The outline of the paper is as follows. Some preliminary estimates needed for the convergence proofs are summarized in Section~\ref{sec-conv-ineq}. 
In Section~\ref{sec-scheme} we define the fully discrete second order scheme and give some of its basic properties, including energy stability and unique solvability. The second order convergence analysis for the nAC equation is presented in Section~\ref{sec-convergence-nAC}. In Section~\ref{sec-consistency} we provide a higher order consistency analysis for the nCH equation, up to order $O (s^3 + h^4 )$. In Section~\ref{sec-convergence} we give the details of the convergence analyses for the nCH equation, in both $\ell^\infty(0, T; \ell^2)$ 
and $\ell^\infty (0, T; \ell^\infty )$ norms. Finally, some numerical results are presented in Section~\ref{sec-numerical results}, that confirm convergence of the schemes. 

	\section{The discrete periodic convolution and useful inequalities}
        \label{sec-conv-ineq}

Here we define a discrete periodic convolution operator on a 2D periodic grid.  We need two spaces: ${\mathcal C}_{m\times n}$ is the space of cell-centered grid (or lattice) functions, and $\mathcal V_{m\times n}$ is the space of vertex-centered grid functions.  The precise definitions can be found in App.~\ref{app-discrete-space}.  The spaces and the following notations have straightforward extensions to three dimensions.  Suppose  $\phi\in \mathcal C_{m\times n}$ is periodic and $f \in \mathcal V_{m\times n}$ is periodic.  Then the discrete convolution operator $\convtwo{f}{\phi}: {\mathcal V}_{m \times n} \times {\mathcal C}_{m\times n} \rightarrow {\mathcal C}_{m\times n}$ is defined component-wise as
	\begin{equation}
\convtwo{f}{\phi} _{i,j} := h^2\sum^{m}_{k=1}\sum^{n}_{l=1} f_{k+\frac{1}{2},l+\frac{1}{2}}\phi_{i-k,j-l} \ .
	\end{equation}
Note very carefully that the order is important in the definition of the discrete convolution $\convtwo{\, \cdot \, }{\, \cdot \, }$. The next result follows from the definition of the discrete convolution and simple estimates.  The proof is omitted.

	\begin{lem}
If $\phi, \psi \in \mathcal{C}_{m \times n}$ are periodic and $f \in \mathcal{V}_{m \times n}$ is periodic and even, \emph{i.e.}, $f_{i+\hf,j+\hf} = f_{-i+\hf,-j+\hf}$, for all $i,j\in \mathbb{Z}$, then
	\begin{equation}
\ciptwo{\phi}{\convtwo{f}{\psi}} = \ciptwo{\psi}{\convtwo{f}{\phi}}  .
	\label{lemma:exchange-2d}
	\end{equation}
If, in addition, $f$ is non-negative, then 
	\begin{equation}
\left|\ciptwo{\phi}{\convtwo{f}{\psi}}\right| \le \convone{f}{{\bf 1}}\left(\frac{\alpha}{2} \ciptwo{\phi}{\phi}+\frac{1}{2\alpha}\ciptwo{\psi}{\psi}\right),  \quad \forall \alpha > 0 . 
 	\label{lemma:discrete_young-2d}
	\end{equation}
	\end{lem}

The following lemma is cited from \cite{guan12a}; the detailed proof can be found there.

	\begin{lem}
	\label{lem:error-4-2d}
Suppose $\phi \in \mathcal{C}_{m \times n}$ is periodic.  Assume that $\mathsf{f} \in C_{\rm per}^\infty(\Omega)$ is even and define its grid restriction via $f_{i+\frac{1}{2},j+\hf} := \mathsf{f}\left(p_{i+\frac{1}{2}},p_{j+\frac{1}{2}}\right)$, so that $f\in\mathcal{V}_{m\times n}$. Then for any $\alpha > 0$, we have
	\begin{equation}
- 2h^2\ciptwo{\convone{f}{\psi}}{\Delta_h \phi} \le \frac{C_2}{\alpha} \nrm{\psi}^2_2 + \alpha \nrm{\nabla_h \phi}^2_2  ,
	\end{equation}
where $C_2$ is a positive constant that depends on $\mathsf{f}$ but is independent of $h$.
	\end{lem}

	\section{The second order convex splitting schemes}\label{sec-scheme}
	
	\subsection{Semi discrete convex splitting schemes}
	The second order (in time) convex splitting scheme for the nAC equation (\ref{pde-non conserve}) and nCH equation (\ref{pde-conserve}) has been proposed in a recent article \cite{guan12b}, which was proven to 
be unconditionally solvable and unconditionally energy stable. These schemes follows the convex properties of the energy~\cite{guan12a}:
	\begin{lem}
	\label{energy-bounds}
There exists a non-negative constant $C$ such that $E(\phi) +C \ge 0$ for all $\phi\in L_{\rm per}^4(\Omega)$.  More specifically,
	\begin{eqnarray}
\frac{1}{8}\norm{\phi}{L^4}^4 &\le& E(\phi) + \frac{\left(\gamma_c-\gamma_e-2\left(\msfJ_e*1\right)\right)^2}{2}|\Omega|  ,
	\\
\frac{1}{2}\norm{\phi}{L^2}^2 &\le& E(\phi) +\frac{\left(\gamma-\gamma_e-2\left(\msfJ_e*1\right)-1\right)^2}{4}|\Omega|   .
	\end{eqnarray}
If $\gamma_e = 0$, then $E(\phi) \ge 0$ for all $\phi\in L_{\rm per}^4(\Omega)$.  Furthermore, the energy (\ref{non-local-energy}) can be written as the difference of convex functionals, \emph{i.e}, $E = E_c-E_e$, where
	\begin{eqnarray}
E_c(\phi) &=& \frac{1}{4}\norm{\phi}{L^4}^4 + \frac{\gamma_c+2\left(\msfJ_c*1\right)}{2}\norm{\phi}{L^2}^2 ,
	\label{E-c} 
	\\
E_e(\phi) &=& \frac{\gamma_e+\msfJ_c*1+\msfJ_e*1}{2}\norm{\phi}{L^2}^2 + \frac{1}{2}\iprd{\phi}{\msfJ*\phi}{L^2}.
	\label{E-e}
	\end{eqnarray}
	\end{lem}
The decomposition above is not unique, but Eqs.~\eqref{E-c} -- \eqref{E-e} will allow us to separate the nonlocal and nonlinear terms, treating the nonlinearity implicitly and the nonlocal term explicitly, without sacrificing numerical stability. 

To motivate the fully discrete scheme that will follow later, we here present semi-discrete version and briefly describe its properties.

A second-order (in time) convex splitting scheme for the nAC equation (\ref{pde-non conserve}) and nCH equation (\ref{pde-conserve}) can be constructed as follows: given $\phi^k\in C^\infty_{\rm per}(\Omega)$, find $\phi^{k+1},\, w^{k+1}\in C^\infty_{\rm per}(\Omega)$ such that
	\begin{eqnarray}
  \phi^{k+1}-\phi^k &=& - M s w^{k+1/2},  \quad  \mbox{(nAC equation)} , 
	\label{nlac-scheme-1}
\\	
\phi^{k+1}-\phi^k &=& s\Delta w^{k+1/2},  \quad  \mbox{(nCH equation)} , 
	\label{nlch-scheme-1}
	\\
w^{k+1/2} &=& \eta \left( \phi^k, \phi^{k+1}  \right) 
  + \left( (\msfJ_c*1)+\gamma_c\right) \phi^{k+1/2} 
   - \left( \msfJ_e*1 +\gamma_e\right) \hat{\phi}^{k+1/2} 
  - \msfJ*\hat{\phi}^{k+1/2}  ,
	\label{nlch-scheme-2}
        \\
&& \eta \left( \phi^k, \phi^{k+1}  \right)     
  = \frac14 \left( (\phi^k)^2 + (\phi^{k+1})^2 \right)  
  \left( \phi^k + \phi^{k+1}  \right) ,   \label{nlch-scheme-3} 
        \\
&& 
  \phi^{k+1/2} = \frac12 \left( \phi^k + \phi^{k+1}  \right) ,  \, \,  
  \hat{\phi}^{k+1/2} = \frac32\phi^k - \frac12 \phi^{k-1} ,  
   \label{nlch-scheme-4}
	\end{eqnarray}
where $s>0$ is the time step size. This scheme respects the convex splitting nature of the energy $E$ given in \eqref{E-c} and \eqref{E-e}: the contribution to the chemical potential corresponding to the convex energy, $E_c$, is treated implicitly, the part corresponding to the concave part, $E_e$, is treated explicitly.  Eyre~\cite{eyre98} is often credited with proposing the convex-splitting methodology for the Cahn-Hilliard and Allen-Cahn equations.  The idea is, however, quite general and can be applied to any gradient flow of an energy that splits into convex and concave parts.  See for example~\cite{wang10a,wang11,wise09,wise10}.  Moreover, the treatment of the nonlinear term, the convex and concave diffusion terms, given by 
(\ref{nlch-scheme-3}), (\ref{nlch-scheme-4}), respectively, follows the methodology in an earlier work~\cite{hu09} to derive a second order accurate convex splitting scheme for the phase field crystal (PFC) model. Other related works can also be found in~\cite{baskaran12,shen12}, etc. 

We have the following \emph{a priori} energy law for the solutions of the second-order scheme (\ref{nlch-scheme-1}) -- (\ref{nlch-scheme-4}). 
The statement for the fully discrete version appears later in Section~\ref{subsec-solavble-stable}. 
Their proof can be found in a recent article \cite{guan12b}. 
	
	\begin{thm}
	\label{thm-energy-stability}
Suppose the energy $E(\phi)$ is defined in Eq.~(\ref{non-local-energy}). For any $s>0$, the second order convex splitting scheme, (\ref{nlac-scheme-1}) or (\ref{nlch-scheme-1}), with (\ref{nlch-scheme-2}) -- (\ref{nlch-scheme-4}), has a unique solution $\phi^{k+1},\, w^{k+1/2}\in C^\infty_{\rm per}(\Omega)$. Moreover, by denoting a pseudo energy
\begin{eqnarray} 
  {\cal E} \left( \phi^k , \phi^{k+1} \right) 
  &=&  E\left(\phi^{k+1}\right)  
     + \frac{\left( (\msfJ_c*1) + \msfJ_e*1+\gamma_c+\gamma_e\right)}{4}\norm{\phi^{k+1}-\phi^k}{L^2}^2    \nonumber 
\\
  &&
  + \frac{1}{4}
  \iprd{\msfJ*\left(\phi^{k+1}-\phi^k\right)}{\phi^{k+1}-\phi^k}{L^2} ,  
  \label{pseudo energy-1}
\end{eqnarray}
we have $\left(\phi^{k+1}-\phi^k,1\right) = 0$ for any $k\ge 1$ and 
	\begin{eqnarray}
  &&
  {\cal E} \left( \phi^k , \phi^{k+1} \right)  + M s \norm{w^{k+1/2}}{L^2}^2  
  \le  {\cal E} \left( \phi^{k-1} , \phi^k \right)  ,  \quad  \mbox{(nAC equation)}  , 
\\	
  &&
  {\cal E} \left( \phi^k , \phi^{k+1} \right)  + s\norm{\nabla w^{k+1/2}}{L^2}^2  
  \le  {\cal E} \left( \phi^{k-1} , \phi^k \right)  ,  \quad  \mbox{(nCH equation)} .   	
	\label{energy-equality}
	\end{eqnarray}
Also note that the remainder term in the pseudo energy (\ref{pseudo energy-1}) is non-negative. This implies that the energy is bounded by the initial energy, 
for any $s > 0$:   
$E\left(\phi^{k+1}\right) \le {\cal E} \left( \phi^k , \phi^{k+1} \right) 
\le {\cal E} \left( \phi^{-1} , \phi^0 \right)  \equiv 
E \left( \phi^0 \right)$ by taking $\phi^{-1} = \phi^0$. 
Therefore, we say that the scheme is unconditionally energy stable.
	\end{thm}

In the sequel, we will propose fully discrete versions of these schemes and  provide the corresponding analysis.

	\subsection{Discrete energy and the fully discrete convex splitting schemes}
	
We begin by defining a fully discrete energy that is consistent with the energy (\ref{non-local-energy}) in continuous space.  In particular, the discrete energy $F:{\mathcal C}_{m \times n}\rightarrow \mathbb{R}$ is given by
	\begin{equation}
F(\phi) := \frac{1}{4}\nrm{\phi}_4^4+\frac{\gamma_c-\gamma_e}{2} \nrm{\phi}_2^2 + \frac{\convone{J}{{\bf 1}}}{2}\nrm{\phi}_2^2 - \frac{h^2}{2}\ciptwo{\phi}{\convone{J}{\phi}}  , 
	\label{discrete-energy}
	\end{equation}
where $J := J_c-J_e$, and $J_\alpha\in\mathcal{V}_{m \times n}$, $\alpha = c,e$,  are defined via the vertex-centered grid  restrictions $\left(J_\alpha\right)_{i+\hf,j+\hf} := \msfJ_\alpha(p_{i+\hf},p_{j+\hf})$.
	\begin{lem}
	\label{discrete-energy-splitting}
Suppose that $\phi \in \mathcal{C}_{m \times n}$ is periodic and define
	\begin{eqnarray}
F_c(\phi) &:=& \frac{1}{4} \nrm{\phi}_4^4 +\frac{ \convone{J_c}{{\bf 1}}+\gamma_c }{2} \nrm{\phi}_2^2, 
	\\
F_e(\phi) &:=&  \frac{ \convone{J_e}{{\bf 1}}+\gamma_e }{2} \nrm{\phi}_2^2 + \frac{h^2}{2}\ciptwo{\phi}{\convone{J}{\phi}} . 
	\end{eqnarray}
Then $F_c$ and $F_e$ are convex, and the gradients of the respective energies are 
	\begin{equation}
\delta_\phi F_c = \phi^3 +\left( \convone{J_c}{{\bf 1}}+\gamma_c\right)\phi \ , \quad \delta_\phi F_e = \left( \convone{J_e}{{\bf 1}}+\gamma_e\right) \phi + \convone{J}{\phi}  .
	\end{equation}
Hence $F$, as defined in (\ref{discrete-energy}), admits  the convex splitting $F = F_c-F_e$.
	\end{lem}
	\begin{proof}
$F_c$ is clearly convex.  To see that $F_e$ is convex, we make use of the estimate \eqref{lemma:discrete_young-2d}, and observe that $\left.\frac{d^2}{ds^2} F_e(\phi+s\psi)\right|_{s=0} \ge 0$, for any periodic $\psi\in\mathcal{C}_{m\times n}$.  The details are suppressed for brevity.

	\end{proof}

We now describe the fully discrete schemes in detail. 
The scheme can be formulated as:  given $\phi^k\in{\mathcal C}_{m \times n}$ periodic, find $\phi^{k+1},\, 
w^{k+1/2} \in {\mathcal C}_{m \times n}$ periodic so that
	\begin{eqnarray}
  \phi^{k+1}-\phi^k &=& - M s w^{k+1/2}  ,    \quad  \mbox{(nAC equation)} , 
  \label{full-disc-cons-flow-nAC}   
\\	
  \phi^{k+1}-\phi^k &=& s \Delta_h w^{k+1/2}  ,    \quad  \mbox{(nCH equation)} , 
  \label{full-disc-cons-flow}   
\\
  w^{k+1/2}  &:=& 
  \eta \left( \phi^k , \phi^{k+1} \right) +\left( \convone{J_c}{{\bf 1}}+\gamma_c\right)\phi^{k+1/2} \nonumber \\
&&- \left( \convone{J_e}{{\bf 1}}+\gamma_e\right) \hat{\phi}^{k+1/2} 
  - \convone{J}{\hat{\phi}^{k+1/2}}  ,
	\end{eqnarray} 
in which $\Delta_h$ is the standard five-point discrete Laplacian operator, 
$\eta \left( \phi^k , \phi^{k+1} \right)$, 
$\phi^{k+1/2}$, $\hat{\phi}^{k+1/2}$ are given by (\ref{nlch-scheme-3}) -- (\ref{nlch-scheme-4}), respectively.

	\subsection{Unconditional solvability and energy stability}
	\label{subsec-solavble-stable}

Now we show that the convexity splitting is translated into solvability and stability for our scheme, both (\ref{full-disc-cons-flow-nAC}) and (\ref{full-disc-cons-flow}). The basic method for the proof of the following result was established in~\cite{guan12b, guan12a} -- see also ~\cite{wang11, wise09,wise10} -- and we therefore omit it for brevity.

	\begin{thm}
	\label{thm:solvability-conserve}
The scheme (\ref{full-disc-cons-flow}) is discretely mass conservative, \emph{i.e.}, $\ciptwo{\phi^{k+1}-\phi^k}{{\bf 1}}=0$, for all $k\ge 0$, and uniquely solvable for any time step size $s>0$.
	\end{thm}

	\begin{lem}
	\label{lp-stable-lemma}
Suppose that $\phi\in {\mathcal C}_{m\times n}$ is periodic and the discrete energy  $F$ is  defined in Eq.~(\ref{discrete-energy}).  There exists a non-negative constant $C$ such that $F(\phi) +C \ge 0$.  More specifically,
	\begin{eqnarray}
\frac{1}{8}\norm{\phi}{4}^4 &\le& F(\phi) + \frac{\left(\gamma_c-\gamma_e-2\convone{J_e}{{\bf 1}}\right)^2}{2}|\Omega|  ,
	\\
\frac{1}{2}\norm{\phi}{2}^2 &\le& F(\phi) +\frac{\left(\gamma_c-\gamma_e-2\convone{J_e}{{\bf 1}}-1\right)^2}{4}|\Omega|   .
	\end{eqnarray}
	\end{lem}

The following result is a discrete version of Theorem~\ref{thm-energy-stability}; 
its proof can be found in \cite{guan12b}. 

	\begin{thm}
	\label{energy-decay-thm}
Suppose the energy $F(\phi)$ is defined in Eq.~(\ref{discrete-energy}), assume $\phi^{k+1},\, \phi^k\in{\mathcal C}_{m \times n}$ are periodic and they are solutions to the scheme (\ref{full-disc-cons-flow-nAC}) (for the nAC equation) or (\ref{full-disc-cons-flow}) (for the nCH equation).  Then for any $s>0$, 
	\begin{eqnarray}
  &&
  {\cal E} \left( \phi^k , \phi^{k+1} \right)  + M s \norm{w^{k+1/2}}{2}^2   
  \le  {\cal E} \left( \phi^{k-1} , \phi^k \right)   , \quad  \mbox{(nAC equation)} , 
\\	
  &&
  {\cal E} \left( \phi^k , \phi^{k+1} \right)  + s\norm{\nabla_h w^{k+1/2}}{2}^2   
  \le  {\cal E} \left( \phi^{k-1} , \phi^k \right)   , \quad  \mbox{(nCH equation)} , 
	\label{discrete-energy-equality}
	\end{eqnarray}
with the discrete pseudo energy ${\cal E} \left( \phi^k , \phi^{k+1} \right)$ 
given by
	\begin{eqnarray}
{\cal E} \left( \phi^k , \phi^{k+1} \right) 
  &=&  F \left(\phi^{k+1}\right)  
     + \frac{\left(  \convone{J_c}{{\bf 1}}  
   +  \convone{J_e}{{\bf 1}}  +\gamma_c+\gamma_e\right)}{4}\norm{\phi^{k+1}-\phi^k}{2}^2   \nonumber 
\\
  &&
   + \frac{h^2}{4}\ciptwo{\convone{J}{\left(\phi^{k+1}-\phi^k\right)}}{\phi^{k+1}-\phi^k} . 
  \label{discrete-pseudo energy-1}
	\end{eqnarray}
Most importantly, the remainder term in the pseudo energy 
(\ref{discrete-pseudo energy-1}) is non-negative. This implies that the energy is bounded by the initial energy, 
for any $s > 0$:   
$F \left(\phi^{k+1}\right) \le {\cal E} \left( \phi^k , \phi^{k+1} \right) 
\le {\cal E} \left( \phi^{-1} , \phi^0 \right)  \equiv 
F \left( \phi^0 \right)$ by taking $\phi^{-1} = \phi^0$. 
	\end{thm}


Putting Lemma~\ref{lp-stable-lemma} and Theorem~\ref{energy-decay-thm} together, we immediately get the following two results. The proof can be 
found in \cite{guan12a}. 
	
	\begin{cor}
	\label{lp-stable-thm-2}
Suppose that $\left\{\phi^k,w^k\right\}_{k=1}^l\in \left[{\mathcal C}_{m \times n}\right]^2$ are a sequence of periodic solution pairs of the scheme (\ref{full-disc-cons-flow-nAC}) (for the nAC equation) or (\ref{full-disc-cons-flow}) (for the nCH equation), with the starting values $\phi^0$. Then, for any $1 \le k \le l$,
	\begin{eqnarray}
\frac{1}{8}\norm{\phi^k}{4}^4 &\le& F\left(\phi^0\right) + \frac{\left(\gamma_c-\gamma_e-2\convone{J_e}{{\bf 1}}\right)^2}{2}|\Omega|  ,
	\\
\frac{1}{2}\norm{\phi^k}{2}^2 &\le& F\left(\phi^0\right) +\frac{\left(\gamma_c-\gamma_e-2\convone{J_e}{{\bf 1}}-1\right)^2}{4}|\Omega|   .
	\end{eqnarray}
	\end{cor}

	\begin{thm}
	\label{lp-stable-thm}
Suppose $\Phi\in C_{\rm per}^r(\Omega)$, where $r\in\mathbb{Z}^+$ is sufficiently large, and set $\phi^0_{i,j} := \Phi(p_i,p_j)$. Suppose that $\left\{\phi^k,w^k\right\}_{k=1}^l\in \left[{\mathcal C}_{m \times n}\right]^2$ are a sequence of periodic solution pairs of the scheme (\ref{full-disc-cons-flow-nAC}) (for the nAC equation) or (\ref{full-disc-cons-flow}) (for the nCH equation), with the starting values $\phi^0$. There exist constants  $C_3,\, C_4,\, C_5 >0$, which are independent of $h$ and $s$, such that 
	\begin{eqnarray}
\max_{1\le k\le l}\nrm{\phi^k}_{4} &\le& C_3 .
	\label{l4-uniform-bound}
	\\
\max_{1\le k\le l}\nrm{\phi^k}_2 &\le& C_4 . 
	\label{l2-uniform-bound}
	\end{eqnarray}
\end{thm}

	\subsection{Convergence result in $\ell^\infty (\ell^2)$ and $\ell^{\infty}(\ell^\infty)$ norm}

We conclude this subsection with the statement of local-in-time error estimates for our second order convex splitting schemes, including both (\ref{full-disc-cons-flow-nAC}) for the nAC equation and (\ref{full-disc-cons-flow}) for the nCH equation, in two dimensions. The detailed proof is given in the next three sections. The extension of the proofs to three dimensions is omitted for the sake of brevity; see Remark~\ref{rem-3D-1} for some of the details.  

The second order convergence of the scheme (\ref{full-disc-cons-flow-nAC}) for the nAC equation is stated in the following theorem. 
 
	\begin{thm}\label{thm:L2-converge-nAC}
Given smooth, periodic initial data $\Phi(x_1,x_2,t=0)$, suppose  the unique, smooth, periodic solution for the nAC equation (\ref{pde-non conserve}) is given by $\Phi(x,y,t)$ on $\Omega$ for $0<t\le T$,  for some $T< \infty$. Define $\Phi_{i,j}^k$ as above and set $e_{i,j}^k := \Phi_{i,j}^k-\phi_{i,j}^k$, where $\phi^k_{i,j}\in {\mathcal C}_{m \times n}$ is $k^{\rm th}$ periodic solution of (\ref{full-disc-cons-flow-nAC}) with $\phi^0_{i,j} := \Phi^0_{i,j}$.  Then, provided $s$ and $h$ are sufficiently small, we have 
	\begin{equation}
\nrm{e^l}_2 \le C\left(h^2+s^2 \right),  \label{converge-nAC-0} 
	\end{equation} 
for all positive integers $l$, such that $l\, s \le T$, where $C>0$ is independent of $h$ and $s$.
	\end{thm}

For the nCH equation, the existence and uniqueness of a smooth, periodic solution to the IPDE \eqref{pde-conserve} with smooth periodic initial data may be established using techniques developed by Bates and Han in~\cite{bates05b, bates05a}.  In the following pages we denote this IPDE solution by $\Phi$.  Motivated by the results of Bates and Han, and based on the assumptions in the introduction, it will be reasonable to assume that 
	\begin{equation}
\nrm{\Phi}_{L^\infty(0,T;L^4)} + \nrm{\Phi}_{L^\infty(0,T;L^\infty)} + \nrm{ \nabla \Phi}_{L^\infty(0,T;L^\infty)} < C  ,  
	\end{equation}
for any $T >0$, and therefore also, using a consistency argument, that
	\begin{equation}
\max_{1\le k\le \ell} \nrm{\Phi^k}_4 + \max_{1\le k\le \ell}\nrm{\Phi^k}_\infty + \max_{1\le k\le \ell}\nrm{ \nabla_h \Phi^k}_\infty < C ,
	\label{IPDE-solution-stabilities}
	\end{equation}
after setting $\Phi^k_{i,j} := \Phi(p_i,p_j,t_k)$, where $C$ is independent of $h$ and $s$ and $t_k=k \cdot s$.  The IPDE solution $\Phi$ is also mass conservative, meaning that, for all $0\le t\le T$, $\int_\Omega \left( \Phi({\bf x},0)  - \Phi({\bf x},t)\right) d{\bf x}= 0$. For our numerical scheme, on choosing $\phi_{i,j}^0 := \Phi_{i,j}^0$, we note that $\ciptwo{\phi^k-\Phi^0}{{\bf 1}} = 0$, for all $k\ge 0$.  But, unfortunately, $\ciptwo{\phi^k-\Phi^k}{{\bf 1}} \ne 0$ in general.  On the other hand, by consistency, 
	\begin{equation}
  \beta (t) = \int_\Omega  \Phi({\bf x},t)  d{\bf x}  
  -  \frac{1}{L_1L_2}h^2\ciptwo{\Phi (t) }{{\bf 1}} 
   -   \left(  \int_\Omega  \Phi^0 ({\bf x})  d{\bf x}  
  -  \frac{1}{L_1L_2}h^2\ciptwo{\Phi^0}{{\bf 1}} \right) ,  
       \end{equation}
for some $C>0$ that is independent of $k$ and $h$. A more 
detailed consistency analysis shows that $\left| \beta (t)  \right|  \le C h^2$, 
$\forall t \ge 0$, and the estimates for all its higher order derivatives are 
available. Then we have 
        \begin{equation} 
\frac{1}{L_1L_2}h^2\ciptwo{\phi^k-\Phi^k}{{\bf 1}} =:  \beta^k 
  = \beta \left( t^k \right) , \quad \left|\beta^k\right| \le Ch^2,
	\end{equation}
for all $1\le k \le l$. 
We set $\tilde{\Phi} ( \cdot, t )  := \Phi ( \cdot, t) +\beta (t)$ and observe $\ciptwo{\phi^k-\tilde{\Phi}^k}{{\bf 1}} = 0$ and also
	\begin{equation}
\max_{1\le k\le \ell} \nrm{\tilde{\Phi}^k}_4 + \max_{1\le k\le \ell}\nrm{\tilde{\Phi}^k}_\infty + \max_{1\le k\le \ell}\nrm{ \nabla_h \tilde{\Phi}^k}_\infty < C .
	\label{IPDE-corrected-solution-stabilities}
	\end{equation}

Finally, the assumptions on the continuous kernel $\msfJ$, specifically \eqref{gamma-0}, and the consistency of the discrete convolution implies 
that $\convone{J_c}{{\bf 1}} + \gamma_c 
  - \convone{J_e}{{\bf 1}} - \gamma_e > 0$. Furthermore, we make 
the following assumption to simplify the convergence analysis: 
	\begin{equation}
	 B_c = \convone{J_c}{{\bf 1}} + \gamma_c ,  \, \, \, 
  B_e = \convone{J_e}{{\bf 1}} + \gamma_e ,  \quad 
  \mbox{and} \, \, B_c - 3 B_e = \alpha_0 > 0 , 
	\label{discrete-alpha-0}
	\end{equation}
for some $\alpha_0$ that is independent of $h$, provided that $h$ is sufficiently small. However, numerical evidence indicates that our scheme converges at the same rate when $\alpha_0 \leq 0$.

	\begin{thm}\label{thm:L2-converge}
Given smooth, periodic initial data $\Phi(x_1,x_2,t=0)$, suppose  the unique, smooth, periodic solution for the IPDE (\ref{pde-conserve}) is given by $\Phi(x,y,t)$ on $\Omega$ for $0<t\le T$,  for some $T< \infty$. Define $\Phi_{i,j}^k$ as above and set $e_{i,j}^k := \Phi_{i,j}^k-\phi_{i,j}^k$, where $\phi^k_{i,j}\in {\mathcal C}_{m \times n}$ is $k^{\rm th}$ periodic solution of (\ref{full-disc-cons-flow}) with $\phi^0_{i,j} := \Phi^0_{i,j}$.  Then, provided $s$ and $h$ are sufficiently small with the linear refinement path constraint $s \le C h$, with $C$ any fixed constant, we have 
	\begin{equation}
\nrm{e^l}_2 \le C\left(h^2+s^2 \right), 
	\end{equation} 
for all positive integers $l$, such that $l\, s \le T$, where $C>0$ is independent of $h$ and $s$.
	\end{thm}

	\begin{thm}
	\label{thm:converge-infty}
Under the assumptions of Theorem~\ref{thm:L2-converge}, we also have optimal order convergence of the numerical solution of the scheme (\ref{full-disc-cons-flow}) in the $\ell^\infty$ norm. Namely, if $s$ and $h$ are sufficiently small, for all positive integers $l$, such that $l\, s \le T$, we have 
	\begin{equation}
\nrm{e^l}_\infty \le C\left(h^2+s^2 \right), 
	\label{converge-infty-1}
	\end{equation} 
where $C>0$ is independent of $h$ and $s$.
	\end{thm}
	
\section{The second order convergence analysis for the nAC equation}  
\label{sec-convergence-nAC}

In this section we provide a proof of Theorem~\ref{thm:L2-converge-nAC}. For the exact solution $\Phi$ of the nAC equation (\ref{pde-non conserve}), a detailed Taylor expansion, combined with a careful Fourier analysis, indicates the following consistency estimate: 
	\begin{eqnarray} 
   && 
\Phi^{k+1} - \Phi^k = - M s \left( \eta \left( \Phi^k , \Phi^{k+1} \right) + B_c \Phi^{k+1/2} 
- B_e \hat{\Phi}^{k+1/2} - \convone{J}{\hat{\Phi}^{k+1/2}}  \right) +s \tau^{k+1}  ,  \label{truncation-error-nAC} 
\\
  && 
  \mbox{with} \quad 
  \Phi^{k+1/2} = \frac12 \left( \Phi^k + \Phi^{k+1}  \right) ,  \, \,  
  \hat{\Phi}^{k+1/2} = \frac32  \Phi^k - \frac12 \Phi^{k-1}  ,  \label{truncation-error-nAC-2}
	\end{eqnarray}
with the local truncation error 
	\begin{equation}
\left|\tau_{i,j}^{k+1}\right| \le C_5 \left(h^2 + s^2 \right)  ,   \quad \forall i, j, k,  
\label{truncation-error-nAC-3} 
	\end{equation}	
and the constant $C_5$ depends only on $T$, $L_1$ and $L_2$. 

We consider the following error function, at a point-wise level:   
	\begin{equation}
e^k_{i,j} := \Phi^k_{i,j} - \phi^k_{i,j}.  
	\label{error function-1}
	\end{equation} 
In turn, subtracting (\ref{full-disc-cons-flow-nAC}) from (\ref{truncation-error-nAC}) leads to 
	\begin{eqnarray}
e^{k+1} - e^k &=& -M s \Bigl( \eta \left(  \Phi^k , \Phi^{k+1} \right) 
  - \eta \left(  \phi^k , \phi^{k+1} \right) 
  + B_c  e^{k+1/2} -  B_e \hat{e}^{k+1/2}   \nonumber 
\\
 &&  \quad 
   - \convone{J}{\hat{e}^{k+1/2}}  \Bigr)
  + s \tau_{i,j}^k ,    \label{converge-nAC-1}
\\
   \mbox{with}  && 
    e^{k+1/2} = \frac12 \left(  e^k + e^{k+1} \right)  ,  \, \, 
    \hat{e}^{k+1/2} = \frac32 e^k - \frac12 e^{k-1} .  \nonumber
	\end{eqnarray}
	
Taking a discrete inner product with (\ref{converge-nAC-1}) by $2 e^{k+1/2} = ( e^{k+1} + e^k)$, summing over $i$ and $j$ implies that 
\begin{eqnarray}
   &&  
\| e^{k+1} \|_2^2 - \| e^k \|_2^2  + 2 B_c s \|  e^{k+1/2} \|_2^2   
   + 2h^2s \ciptwo{  \eta ( \Phi^k , \Phi^{k+1} ) 
  - \eta (  \phi^k , \phi^{k+1} )  }{e^{k+1/2} }    \nonumber 
\\
  &\le& 
     2sh^2 \ciptwo{\tau^{k+1}}{ e^{k+1/2}} 
     + 2 B_e s h^2  \ciptwo{ \hat{e}^{k+1/2}}{ e^{k+1/2}} 
  + 2 h^2s\ciptwo{\convone{J}{ \hat{e}^{k+1/2}}}
  {e^{k+1/2}}  .   \label{converge-nAC-2}
	\end{eqnarray}
The term associated with the local truncation error could be bounded with an application of Caucy inequality: 
\begin{eqnarray} 
  2 h^2 \ciptwo{\tau^{k+1}}{ e^{k+1/2}}   
  \le \| \tau^{k+1} \|_2^2 + \| e^{k+1/2} \|_2^2 
  \le   \| \tau^{k+1} \|_2^2 + \frac12 ( \| e^{k+1} \|_2^2  + \| e^k \|_2^2 ) . 
   \label{converge-nAC-3}
\end{eqnarray}  	
The concave term could be bounded in a straightforward way: 
\begin{eqnarray} 
  2 B_e  h^2  \ciptwo{ \hat{e}^{k+1/2}}{ e^{k+1/2}}  
 &\le& B_e ( \| \hat{e}^{k+1/2} \|_2^2 + \| e^{k+1/2} \|_2^2  \nonumber 
\\
  &\le&
   \frac{B_e}{2} ( \| e^{k+1} \|_2^2  + 10 \| e^k \|_2^2 + \| e^{k-1} \|_2^2 )  . 
   \label{converge-nAC-4}
\end{eqnarray} 
The term associated with the convolution could be analyzed with the help of (\ref{lemma:discrete_young-2d}): 
\begin{eqnarray} 
  2 h^2 \ciptwo{\convone{J}{ \hat{e}^{k+1/2}}}{e^{k+1/2}}    
 &\le& B_J ( \| \hat{e}^{k+1/2} \|_2^2 + \| e^{k+1/2} \|_2^2  \nonumber 
\\
  &\le&
   \frac{B_J}{2} ( \| e^{k+1} \|_2^2  + 10 \| e^k \|_2^2 + \| e^{k-1} \|_2^2 )  , 
   \label{converge-nAC-5}
\end{eqnarray} 
with $B_J =  \convone{J_c}{{\bf 1}} + \convone{J_e}{{\bf 1}}$. 

The rest work is focused on the analysis for the term associated with the nonlinear error. We begin with the following decomposition: 
\begin{eqnarray} 
  &&
 \eta ( \Phi^k , \Phi^{k+1} ) - \eta (  \phi^k , \phi^{k+1} )  
 = {\cal NLE}_1 + {\cal NLE}_2 + {\cal NLE}_3 ,  \quad \mbox{with}  \label{converge-nAC-6}  
\\
  &&
 {\cal NLE}_1 = \frac12 ( (\phi^{k+1})^2 + (\phi^k)^2 ) e^{k+1/2}  ,  \quad 
 {\cal NLE}_2  = \frac14 ( \Phi^{k+1} + \phi^{k+1} ) ( \Phi^{k+1} + \Phi^k ) e^{k+1} ,  \nonumber 
\\
  && 
 {\cal NLE}_3 = \frac14 ( \Phi^k + \phi^k ) ( \Phi^{k+1} + \Phi^k ) e^k  .  \nonumber 
\end{eqnarray} 
The following estimate is available for the term associated with ${\cal NLE}_1$: 
\begin{eqnarray} 
  - 2 h^2  \ciptwo{ {\cal NLE}_1  }{e^{k+1/2} }  
  \le - h^2  \ciptwo{ (\phi^{k+1})^2  }{(e^{k+1/2})^2 }     
   - h^2  \ciptwo{ (\phi^k)^2  }{(e^{k+1/2})^2 }   .   \label{converge-nAC-7} 
\end{eqnarray} 
For the term associated with ${\cal NLE}_2$, we have
\begin{eqnarray} 
  &&
       2 h^2  \ciptwo{ \frac14 \Phi^{k+1} ( \Phi^{k+1} + \Phi^k ) e^{k+1}  }{e^{k+1/2} }  
   \le \frac12 \| \Phi^{k+1} \|_\infty ( \| \Phi^{k+1} \|_\infty  + \| \Phi^k \|_\infty ) 
   \| e^{k+1} \|_2 \cdot \| e^{k+1/2} \|_2  \nonumber 
\\
  &&  \qquad 
   \le  C  \| e^{k+1} \|_2 \cdot \| e^{k+1/2} \|_2    
   \le C ( \| e^{k+1} \|_2^2 + \| e^k \|_2^2 ) ,  \label{converge-nAC-8-1} 
\\
     &&
       2 h^2  \ciptwo{ \frac14 \phi^{k+1} ( \Phi^{k+1} + \Phi^k ) e^{k+1}  }{e^{k+1/2} } 
   = \frac12 h^2  \ciptwo{( \Phi^{k+1} + \Phi^k ) e^{k+1}  }{\phi^{k+1} e^{k+1/2} }   \nonumber 
\\
  &&  \qquad        
   \le \frac12 ( \| \Phi^{k+1} \|_\infty  + \| \Phi^k \|_\infty ) 
   \| e^{k+1} \|_2 \cdot \| \phi^{k+1} e^{k+1/2} \|_2  
   \le C \| e^{k+1} \|_2 \cdot \| \phi^{k+1} e^{k+1/2} \|_2   
   \nonumber 
\\
  &&  \qquad 
   \le  C  \| e^{k+1} \|_2^2 + \| \phi^{k+1} e^{k+1/2} \|_2    
   = C \| e^{k+1} \|_2^2  + h^2  \ciptwo{ (\phi^{k+1})^2  }{(e^{k+1/2})^2 } ,  
   \label{converge-nAC-8-2}    
   \end{eqnarray} 
with the Cauchy inequality applied in the last step. Consequently, we arrive at 
\begin{eqnarray} 
  - 2 h^2  \ciptwo{ {\cal NLE}_2  }{e^{k+1/2} }  
  \le C ( \| e^{k+1} \|_2^2  + \| e^k \|_2^2 ) 
  + h^2  \ciptwo{ (\phi^{k+1})^2  }{(e^{k+1/2})^2 } .   \label{converge-nAC-8-3} 
\end{eqnarray} 
The estimate for the term associated with ${\cal NLE}_3$ is similar; the details are left to the interested readers: 
\begin{eqnarray} 
  - 2 h^2  \ciptwo{ {\cal NLE}_3  }{e^{k+1/2} }  
  \le C ( \| e^{k+1} \|_2^2  + \| e^k \|_2^2 ) 
  + h^2  \ciptwo{ (\phi^k)^2  }{(e^{k+1/2})^2 } .   \label{converge-nAC-9} 
\end{eqnarray} 
In turn, a combination of (\ref{converge-nAC-7}), (\ref{converge-nAC-8-3}) and (\ref{converge-nAC-9}) into (\ref{converge-nAC-6}) yields the inner product associated with the nonlinear error: 
\begin{eqnarray} 
  - 2 h^2  \ciptwo{  \eta ( \Phi^k , \Phi^{k+1} ) - \eta (  \phi^k , \phi^{k+1} )  }{e^{k+1/2} }  
  \le C ( \| e^{k+1} \|_2^2  + \| e^k \|_2^2 ) .   \label{converge-nAC-10} 
\end{eqnarray} 

  Finally, a substitution of (\ref{converge-nAC-3})-(\ref{converge-nAC-5}) and (\ref{converge-nAC-10}) into (\ref{converge-nAC-2}) indicates that 
\begin{eqnarray}   
\| e^{k+1} \|_2^2 - \| e^k \|_2^2  + 2 B_c s \|  e^{k+1/2} \|_2^2   
   \le C s ( \| e^{k+1} \|_2^2  + \| e^k \|_2^2 + \| e^{k-1} \|_2^2 )  
   + s \| \tau^{k+1} \|_2^2 .   \label{converge-nAC-11}
\end{eqnarray}
An application of discrete Gronwall inequality implies the $\ell^\infty (0,T; \ell^2)$ convergence 
estimate (\ref{converge-nAC-0}), using the local truncation error bound (\ref{truncation-error-nAC-3}). The proof of Theorem~\ref{thm:L2-converge-nAC} is complete.

\begin{rem}   \label{rem:converge-nAC-1}
  In the convergence proof for Theorem~\ref{thm:L2-converge-nAC}, the decomposition (\ref{converge-nAC-6}) has played a key role in the nonlinear error estimates. Due to the well-posed nonlinear inner product in (\ref{converge-nAC-7}), the degree of nonlinearity of the  two other nonlinear inner products could be perfectly controlled, with only the maximum norm bound of the exact solution $\Phi$ needed. 

  As a result of this technique, an estimate for the maximum norm of the numerical solution is avoided, which usually has to be obtained in the nonlinear convergence analysis. Because of this fact, an inverse inequality is not needed in the presented analysis, and the $\ell^\infty (0,T; \ell^2)$ convergence for the nAC equation turns out to be unconditional, i.e., no scaling law between the time step size $s$ and the grid size $h$ is required for the desired convergence result. 
\end{rem} 

\begin{rem}  \label{rem:converge-nAC-2}  
  We also observe that the technical assumption (\ref{discrete-alpha-0}) (for the physical parameters) is not required in the convergence analysis for the nAC equation. 
\end{rem}

\section{Higher order consistency analysis of (\ref{full-disc-cons-flow}) for the nCH equation: asymptotic expansion of the numerical solution}
\label{sec-consistency}

For simplicity of presentation, we denote 
\begin{equation} 
    \tilde{\Phi}^{k+1/2} = \frac12 \left( \tilde{\Phi}^k 
   + \tilde{\Phi}^{k+1}  \right) ,  \, \,  
  \hat{\tilde{\Phi}}^{k+1/2} = \frac32  \tilde{\Phi}^k 
  - \frac12 \tilde{\Phi}^{k-1} . 
\end{equation} 
By consistency, the IPDE solution $\tilde{\Phi}$ solves the discrete equation
	\begin{equation} \label{truncation-error}
\tilde{\Phi}^{k+1} - \tilde{\Phi}^k = s\Dh  \left( \eta \left( \tilde{\Phi}^k , 
\tilde{\Phi}^{k+1} \right) + B_c \tilde{\Phi}^{k+1/2} 
- B_e \hat{\tilde{\Phi}}^{k+1/2} 
  - \convone{J}{\hat{\tilde{\Phi}}^{k+1/2}}  \right) +s \tau^{k+1}  ,
	\end{equation}
where the local truncation error $\tau^{k+1}$ satisfies
	\begin{equation}
\left|\tau_{i,j}^{k+1}\right| \le C_6 \left(h^2 + s^2 \right)  ,
	\end{equation}
for all $i$, $j$, and $k$ for some $C_6\ge0$ that depends only on $T$, $L_1$ and $L_2$.


Meanwhile, it is observed that the leading local truncation error in (\ref{truncation-error}) will not be enough to recover an a-priori 
$W^{1,\infty}$ bound for the 
numerical solution, needed in the stability and convergence analysis. 
To remedy this, we use a higher order consistency analysis, via a perturbation argument, to recover such a bound in later analysis. 
In more detail, we need to construct supplementary fields, $\Phi^1_h$, $\Phi^1_s$, and $\hat{\Phi}$, satisfying
	\begin{equation}
\hat{\Phi} = \tilde{\Phi} + h^2 \Phi_{h,1} + s^2 \Phi_{s,1} ,  
	\label{consistency-1}
	\end{equation} 
so that a higher $O (s^3 + h^4)$ consistency is satisfied with the given numerical scheme (\ref{full-disc-cons-flow}).  The constructed fields $\Phi_{h,1}$, $\Phi_{s,1}$, which will be found using a perturbation expansion, will depend solely on the exact solution $\Phi$.

The following truncation error analysis for the spatial discretization can be obtained by using a straightforward Taylor expansion for the exact solution:  
	\begin{eqnarray} 
\partial_t \tilde{\Phi} = \Delta_h  \left(  \tilde{\Phi}^3 + \left( \convone{J}{{\bf 1}} + \gamma_c - \gamma_e  \right) \tilde{\Phi} - \convone{J}{\tilde{\Phi}}  \right)  + h^2 \g^{(0)} + O(h^4) , \quad \forall\ (i,j) .
	\label{consistency-2}
	\end{eqnarray} 
Here the spatially discrete function $\g^{(0)}$ is smooth enough in the sense that its discrete derivatives are bounded.  Also note that there is no $O(h^3)$ truncation error term, due to the fact that the  centered difference used in the spatial discretization  gives local truncation errors with only even order terms,  $O(h^2)$, $O(h^4)$, etc.  

The spatial correction function $\Phi^1_h$ is given by solving the following equation:  
	\begin{eqnarray} 
\partial_t \Phi_{h,1} = \Delta_h  \left(  3 \tilde{\Phi}^2 \Phi_{h,1}  + \left( \convone{J}{{\bf 1}} + \gamma_c - \gamma_e  \right) \Phi_{h,1} - \convone{J}{\Phi_{h,1}}  \right)  - \g^{(0)} , \quad \forall\ (i,j) .   
	\label{consistency-3}
	\end{eqnarray} 
Existence of a solution of the above linear system of ODEs is a standard exercise. Note that the solution depends only on the exact solution, $\Phi$. In addition, the  divided differences of $\Phi_{h,1}$ of various orders are bounded. 

Now, we define 
	\begin{equation} 
\Phi^*_h := \tilde{\Phi} + h^2 \Phi_{h,1} .
	\label{consistency-4}
	\end{equation} 
A combination of (\ref{consistency-2}) and (\ref{consistency-3}) leads to the fourth order local truncation error for $\Phi^*_h$: 
	\begin{eqnarray} 
\partial_t \Phi^*_h = \Delta_h  \left(  \left( \Phi^*_h \right)^3 + \left(\convone{J}{{\bf 1}} + \gamma_c - \gamma_e  \right) \Phi^*_h   - \convone{J}{\Phi^*_h}  \right)  + O(h^4) , \quad \forall\ (i,j),  
	\label{consistency-5}
	\end{eqnarray} 
for which the following estimate was used: 
	\begin{equation} 
\left( \Phi^*_h \right)^3 = \left( \tilde{\Phi} + h^2 \Phi_{h,1} \right)^3 
= \tilde{\Phi}^3 + 3 h^2 \tilde{\Phi}^2 \Phi_{h,1} + O (h^4) .  
	\label{consistency-6}
	\end{equation} 
We remark that the above derivation is valid since  all $O(h^2)$ terms cancel in the expansion.

Regarding the temporal correction term, we observe that the application  of the second order convex splitting scheme (\ref{full-disc-cons-flow}) for the  profile $\Phi^*_h$ gives
	\begin{eqnarray} 
\frac{ \left( \Phi^*_h \right)^{k+1} - \left( \Phi^*_h \right)^k}{s}  &=& \Delta_h  \Biggl(  \eta \left( ( \Phi^*_h )^k , ( \Phi^*_h )^{k+1}  \right) 
  + B_c \left( \Phi^*_h \right)^{k+1/2}  
  - B_e  \hat{ \Phi^*_h }^{k+1/2}  \nonumber 
\\
  && 
    - \convone{J}{\hat{ \Phi^*_h }^{k+1/2}}  \Biggr) + s^2 \h^{(1)} 
  + O (s^3) + O(h^4)  ,    \quad \forall \ (i,j) ,  
	\label{consistency-7} 
\\
  \mbox{with}  && 
   \left( \Phi^*_h \right)^{k+1/2} 
  = \frac12 \left( \left( \Phi^*_h \right)^k 
  + \left( \Phi^*_h \right)^{k+1} \right)  ,  \, \, 
    \hat{ \Phi^*_h }^{k+1/2} 
  = \frac32 \left( \Phi^*_h \right)^k 
  - \frac12 \left( \Phi^*_h \right)^{k-1} ,  \nonumber 
	\end{eqnarray} 
at any grid point $(i,j)$.   
In turn, the first order temporal correction function  $\Phi_{s,1}$ is given by the solution of the following system of linearized ordinary differential equations  
	\begin{eqnarray} 
\partial_t \Phi_{s,1}  &=& \Delta_h  \Bigl(  
  3 \left( \Phi^*_h \right)^2 \Phi_{s,1} + \left( \convone{J_c}{{\bf 1}}  - \convone{J_e}{{\bf 1}} + \gamma_c  - \gamma_e  \right) \Phi_{s,1}  
  - \convone{J}{\Phi_{s,1}}  \Bigr)  - \h^{(1)}  . 
	\label{consistency-8}
	\end{eqnarray} 
Again, the solution of (\ref{consistency-8}), which exists and is unique, depends solely on the profile $\Phi^*_h$ and is smooth enough in the sense that its  divided differences of various orders are bounded.  
Similar to (\ref{consistency-7}), an application of the second order convex splitting scheme to $\Phi_{s,1}$ reads 
	\begin{eqnarray} 
\frac{ \left( \Phi_{s,1} \right)^{k+1} - \left( \Phi_{s,1} \right)^k}{s}  
 &=& \Delta_h  \Biggl(  \frac12 \left( ( \Phi^*_h )^k 
  + ( \Phi^*_h )^{k+1} \right)  
  \left(  ( \Phi^*_h )^k \Phi_{s,1}^k 
  + ( \Phi^*_h )^{k+1} \Phi_{s,1}^{k+1}  \right)  
  \nonumber 
\\
  && \quad \, \, 
  + \frac14 \left( \Phi_{s,1}^k +  \Phi_{s,1}^{k+1} \right)  
  \left(  \left( ( \Phi^*_h )^k \right)^2  
  + \left( ( \Phi^*_h )^{k+1}  \right)^2  \right)  
  \nonumber 
\\
  && 
  + B_c \Phi_{s,1}^{k+1} 
  - B_e \hat{\Phi_{s,1}}^{k+1/2} 
   - \convone{J}{\hat{ \Phi_{s,1} }^{k+1/2}}  \Biggr)
	\nonumber 
	\\
&&   
  - ( \h^{(1)} )^k + s (\h^{(2)} )^k + O (s^2) + O(s h^2) ,  \quad \forall \ (i,j) , 
     \label{consistency-9}
\\
   \mbox{with}  && 
    \Phi_{s,1}^{k+1/2} 
  = \frac12 \left(  \Phi_{s,1}^k + \Phi_{s,1}^{k+1} \right)  ,  \, \, 
    \hat{ \Phi_{s,1} }^{k+1/2} 
  = \frac32 \Phi_{s,1}^k - \frac12 \Phi_{s,1}^{k-1} .  \nonumber
	\end{eqnarray}     

Therefore, a combination of (\ref{consistency-7}) and  
(\ref{consistency-9}) shows that 
	\begin{eqnarray} 
\frac{ \hat{\Phi}^{k+1} - \hat{\Phi}^k}{s} &=& \Delta_h  \Bigl(  
 \eta \left(  \hat{\Phi}^k , \hat{\Phi}^{k+1} \right) 
  + B_c \hat{\Phi}^{k+1/2} 
  - B_e \hat{\hat{\Phi}}^{k+1/2}
	\nonumber 
	\\
&& - \convone{J}{\hat{ \hat{\Phi} }^{k+1/2}}  \Bigr) 
  + O (s^3 + s^3 h^2 + h^4) , 
  \quad \forall \ (i,j) ,  
	\label{consistency-10}
\\
   \mbox{with}  && 
    \hat{\Phi}^{k+1/2} 
  = \frac12 \left(  \hat{\Phi}^k + \hat{\Phi}^{k+1} \right)  ,  \, \, 
    \hat{ \hat{\Phi} }^{k+1/2} 
  = \frac32 \hat{\Phi}^k - \frac12 \hat{\Phi}^{k-1} ,  \nonumber
	\end{eqnarray}    
in which the construction (\ref{consistency-1}) for the approximate 
solution $\hat{\Phi}$ is recalled and we have used the following estimate 
\begin{eqnarray} 
  &&
    \eta \left(  \hat{\Phi}^k , \hat{\Phi}^{k+1}  \right) 
  =  \eta \left(  (\Phi^*_h )^k + s^2 \Phi_{s,1}^k , 
    (\Phi^*_h )^{k+1} + s^2 \Phi_{s,1}^{k+1}  \right)   \nonumber 
\\
  &=& 
    \eta \left(  (\Phi^*_h )^k ,  (\Phi^*_h )^{k+1}   \right)  
  + \frac12 s^2 \left( ( \Phi^*_h )^k 
  + ( \Phi^*_h )^{k+1} \right)  
  \left(  ( \Phi^*_h )^k \Phi_{s,1}^k 
  + ( \Phi^*_h )^{k+1} \Phi_{s,1}^{k+1}  \right)  
  \nonumber 
\\
  && \quad \, \, 
  + \frac14 s^2 \left( \Phi_{s,1}^k +  \Phi_{s,1}^{k+1} \right)  
  \left(  \left( ( \Phi^*_h )^k \right)^2  
  + \left( ( \Phi^*_h )^{k+1}  \right)^2  \right)   + O (s^4)  .  
  \label{consistency-11}
\end{eqnarray}

\begin{rem} 
For both correction terms, $\Phi_{h,1}$ and $\Phi_{s,1}$, a trivial initial 
data are taken: $\Phi_{h,1} ( \cdot, t=0) \equiv 0$, 
$\Phi_{s,1} ( \cdot, t=0)  \equiv 0$, as given by (\ref{consistency-3}), 
(\ref{consistency-8}), respectively. In addition, both equations are mass 
conserving at a discrete level. Then we conclude that 
\begin{eqnarray} 
     \phi^0 \equiv  \hat{\Phi}^0 ,  \qquad  
     \ciptwo{\phi^k-\tilde{\Phi}^k}{{\bf 1}} = 0  ,  \quad 
   \forall k \ge 0 .    \label{consistency-12}
\end{eqnarray}
These two properties will be used in later analysis. 
\end{rem}

\section{Convergence proof for the nCH equation} 
\label{sec-convergence}

As stated earlier, the purpose of the higher order expansion (\ref{consistency-1}) is to obtain a $W^{1,\infty}$ bound of  the error function via its $L^2$ norm in higher order accuracy  by utilizing an inverse inequality in spatial discretization,  which will be shown below.  A detailed analysis shows that 
	\begin{eqnarray}
   \nrm{  \hat{\Phi} - \Phi }_\infty 
  +  \nrm{  \nabla_h \left( \hat{\Phi} - \Phi \right) }_\infty  
   \le C ( s^2 + h^2 )  , 
	\label{converge-1}
	\end{eqnarray}
since $\nrm{\Phi_{h,1} }_\infty, \nrm{\nabla_h \Phi_{h,1} }_\infty , 
\nrm{\Phi_{s,1} }_\infty , \nrm{\nabla_h \Phi_{s,1} }_\infty \le C$.  Subsequently, the following error function is considered:  
	\begin{equation}
\hat{e}^k_{i,j} := \hat{\Phi}^k_{i,j} - \phi^k_{i,j}.  
	\label{error function-2}
	\end{equation} 
In other words, instead of a direct comparison between the numerical solution $\phi$ and the exact solution $\Phi$ (or $\tilde{\Phi}$), we estimate the error between the numerical solution and the constructed solution to obtain a higher order convergence in the $\nrm{\, \cdot\, }_2$ norm, which follows the technique originally proposed in~\cite{strang64}. 
Subtracting (\ref{full-disc-cons-flow}) from (\ref{consistency-10}) yields
	\begin{eqnarray}
\hat{e}^{k+1} - \hat{e}^k &=& s\Dh\Bigl( 
  \eta \left(  \hat{\Phi}^k , \hat{\Phi}^{k+1} \right) 
  - \eta \left(  \phi^k , \phi^{k+1} \right) 
  + B_c \hat{e}^{k+1/2} 
  -  B_e \hat{\hat{e}}^{k+1/2}   \nonumber 
\\
  &&  \quad 
   - \convone{J}{\hat{ \hat{e} }^{k+1/2}}  \Bigr)
  + s  \hat{\tau}_{i,j}^k ,  \quad \left| \hat{\tau}_{i,j}^k \right| 
  \le C ( s^3 + h^4 ) ,   
   \label{converge-2}
\\
   \mbox{with}  && 
    \hat{e}^{k+1/2} 
  = \frac12 \left(  \hat{e}^k + \hat{e}^{k+1} \right)  ,  \, \, 
    \hat{ \hat{e} }^{k+1/2} 
  = \frac32 \hat{e}^k - \frac12 \hat{e}^{k-1} .  \nonumber
	\end{eqnarray}

\subsection{Preliminary error estimates for linear terms} 

        \begin{prop}
	\label{estimate-linear}
  We have 
	\begin{eqnarray} 
  - 2 h^2 \ciptwo{\convone{J}{ \hat{ \hat{e} }^{k+1/2}  }}{\Dh \hat{e}^{k+1/2}} \le \frac{C_7}{\alpha} \left( \nrm{\hat{e}^{k}}^2_2 
  +  \nrm{\hat{e}^{k-1}}^2_2  \right) 
  +  \alpha \nrm{\nabla_h \hat{e}^{k+1/2}}^2_2 ,  \quad \forall \alpha > 0 , 
  \label{est-linear-1}
\\
  - 2 h^2 \ciptwo{ \hat{ \hat{e} }^{k+1/2} }{\Dh \hat{e}^{k+1/2}} 
  \le  - \left( \nrm{ \nabla_h \hat{e}^{k+1} }_2^2   
  - \nrm{ \nabla_h \hat{e}^{k} }_2^2  \right) 
  + 5  \nrm{ \nabla_h \hat{e}^{k+1/2} }_2^2  
  +   \nrm{ \nabla_h \hat{e}^{k-1/2} }_2^2 ,  
  \label{est-linear-2}
	\end{eqnarray} 
	\end{prop}

\begin{proof} 
The first inequality (\ref{est-linear-1}) is a direct application of
Lemma~\ref{lem:error-4-2d} and Cauchy inequality:  
	\begin{eqnarray}
- 2 h^2 \ciptwo{\convone{J}{ \hat{ \hat{e} }^{k+1/2} }}{\Dh \hat{e}^{k+1/2}} 
  &\le& \frac{s C_2}{\alpha} \nrm{ \hat{ \hat{e} }^{k+1/2} }^2_2 
  + \alpha \nrm{\nabla_h \hat{e}^{k+1/2}}^2_2   \nonumber 
\\
  &\le& 
     \frac{C_2}{\alpha} \left( 3 \nrm{ \hat{e}^k }^2_2  
  + \nrm{ \hat{e}^{k-1} }^2_2  \right) 
  + \alpha \nrm{\nabla_h \hat{e}^{k+1/2}}^2_2 . 
	\end{eqnarray}

For the second inequality (\ref{est-linear-2}), we start from the summation 
by parts: 
\begin{eqnarray} 
  - 2 h^2 \ciptwo{ \hat{ \hat{e} }^{k+1/2} }{\Dh \hat{e}^{k+1/2}}  
  =  2 h^2 \ciptwo{ \nabla_h \hat{ \hat{e} }^{k+1/2} }
   {\nabla_h \hat{e}^{k+1/2}} 
  =  h^2 \ciptwo{ \nabla_h \hat{ \hat{e} }^{k+1/2} }
  {\nabla_h \left( \hat{e}^k + \hat{e}^{k+1} \right) } .    
  \label{est-linear-3}
	\end{eqnarray} 
Meanwhile, the term $\hat{ \hat{e} }^{k+1/2}$ can be rewritten as 
\begin{eqnarray} 
  \hat{ \hat{e} }^{k+1/2}  
  = \frac32 \hat{e}^k - \frac12 \hat{e}^{k-1}  
  = - \left( \hat{e}^{k+1} - \hat{e}^k \right) 
  + 2 \hat{e}^{k+1/2} - \hat{e}^{k-1/2}  ,  
  \label{est-linear-4}
\end{eqnarray}
which in turn gives the following estimate: 
\begin{eqnarray} 
  &&
  2 h^2 \ciptwo{ \nabla_h \hat{ \hat{e} }^{k+1/2} }
   {\nabla_h \hat{e}^{k+1/2}}    \nonumber 
\\
  &=& 
   - h^2 \ciptwo{ \nabla_h \left( \hat{e}^{k+1} - \hat{e}^k \right)  }
   {\nabla_h \left( \hat{e}^{k+1} + \hat{e}^k \right) }  
  + 4 \nrm{ \nabla_h \hat{e}^{k+1/2} }_2^2 
  - 2 h^2 \ciptwo{ \nabla_h \hat{e}^{k+1/2} }
   {\nabla_h \hat{e}^{k-1/2} }    \nonumber
\\
  &=& 
   - \left( \nrm{ \nabla_h \hat{e}^{k+1} }_2^2   
  - \nrm{ \nabla_h \hat{e}^{k} }_2^2  \right) 
  + 4 \nrm{ \nabla_h \hat{e}^{k+1/2} }_2^2 
  - 2 h^2 \ciptwo{ \nabla_h \hat{e}^{k+1/2} }
   {\nabla_h \hat{e}^{k-1/2} }    \nonumber 
\\
  &\le&  
   - \left( \nrm{ \nabla_h \hat{e}^{k+1} }_2^2   
  - \nrm{ \nabla_h \hat{e}^{k} }_2^2  \right) 
  + 5  \nrm{ \nabla_h \hat{e}^{k+1/2} }_2^2  
  +   \nrm{ \nabla_h \hat{e}^{k-1/2} }_2^2 .  
  \label{est-linear-5}
\end{eqnarray}
Finally, its combination with (\ref{est-linear-3}) results in 
(\ref{est-linear-2}). The proof of Proposition~\ref{estimate-linear} is complete. 
\end{proof}

\subsection{Preliminary nonlinear error estimates}  

  The $W^{1, \infty}$ bound for the constructed approximate solution 
$\hat{\Phi}$ is guaranteed by the regularity of the exact solution $\Phi$ 
(and henceforth $\tilde{\Phi}$ and the correction terms $\Phi_{h,1}$, 
$\Phi_{s,1}$), at any time step. Similarly, its divided difference in time is also 
bounded point-wise, which comes from the regularity in time for the constructed solution. For the numerical solution $\phi$, its 
global in time $\ell^4$ has been derived in Theorem~\ref{lp-stable-thm}. 
Moreover, to carry out the error estimate for the nonlinear term, 
we need to make an a-priori $W^{1,\infty}$ assumption for the numerical solution at time step $t^k$, and use the $O (s^3 + h^4)$ order convergence 
in $\ell^2$ to recover such an assumption at the next time step $t^{k+1}$. 

	\begin{prop}
	\label{estimate-nonlinear}
Suppose $\phi^j,\,  \hat{\Phi}^j \in \mathcal{C}_{m\times n}$, are periodic with equal means, \emph{i.e.}, 
$\ciptwo{\phi^j-\hat{\Phi}^j}{{\bf 1}}=0$, $j=k, k+1$, and satisfying
	\begin{eqnarray}
  && 
  \nrm{\hat{\Phi}^k}_4 + \nrm{\hat{\Phi}^k}_{\infty} 
+ \nrm{\nabla_h \hat{\Phi}^k }_{\infty} 
  + \nrm{\hat{\Phi}^{k+1} }_4 + \nrm{\hat{\Phi}^{k+1} }_{\infty} 
+ \nrm{\nabla_h \hat{\Phi}^{k+1} }_{\infty} \le  C_{0},  \label{a priori-1} 
\\
  && 
  \nrm{\frac{\hat{\Phi}^{k+1} - \hat{\Phi}^k }{s} }_{\infty}  \le  C_{0},  
  \label{a priori-2} 
\\
  && 
  \nrm{\phi^k}_4 \le C_{0},  \, \,  \nrm{\phi^{k+1} }_4 \le C_{0},  
   \label{a priori-3}   
\\
  &&
  \nrm{\phi^k}_{\infty} + \nrm{\nabla_h \phi^k }_{\infty} \le  C_{0} , 
   \label{a priori-4} 
	\end{eqnarray}
where $C_{0}$ is an $s$, $h$-independent positive constant. Then, there exists a positive constant $C_1$, which depends on $C_0$ but is independent of 
$s$ and $h$, such that
	\begin{eqnarray}
  &&
2h^2\ciptwo{  \eta \left(  \hat{\Phi}^k , \hat{\Phi}^{k+1} \right) 
  - \eta \left(  \phi^k , \phi^{k+1} \right)  }{\Delta_h 
  \hat{e}^{k+1/2} }   \nonumber 
\\
  &\le&
      \frac{C_1}{\alpha^3} \left(  \nrm{ \hat{e}^k }_2^2  
   + \nrm{ \hat{e}^{k+1} }_2^2  \right)
   + \alpha \nrm{\nabla_h  \hat{e}^{k+1/2}  }_2^2   \nonumber 
\\
  &&  
   + \frac{C_1}{\alpha} s^2 \nrm{\nabla_h  \hat{e}^k  }_2^2
   + \frac{C_1}{\alpha} \left( 1 + \nrm{ \phi^{k+1} }_\infty^2 \right)  
    \nrm{ \hat{e}^k }_2^2 ,  \quad  \forall \alpha > 0 .  
  \label{est-nonlinear-1}
	\end{eqnarray}
	\end{prop}

\begin{proof} 
 For simplicity of presentation, we denote 
${\cal NLE}^k = \eta \left(  \hat{\Phi}^k , \hat{\Phi}^{k+1} \right) 
  - \eta \left(  \phi^k , \phi^{k+1} \right)$. 
A direct application of summation by parts reveals that 
\begin{eqnarray}
  &&
2h^2\ciptwo{  {\cal NLE}^k  }{\Delta_h 
  \hat{e}^{k+1/2} }    
   = - 2h^2\ciptwo{  \nabla_h \left( {\cal NLE}^k \right) }
  {\nabla_h \hat{e}^{k+1/2} }    
  = {\cal NLI}_1^k +  {\cal NLI}_2^k ,   \quad 
  \mbox{with} \label{est-nonlinear-2}   
\\
  &&
   {\cal NLI}_1^k = - 2h^2\ciptwo{  D_x \left( {\cal NLE}^k \right) }
  {D_x \hat{e}^{k+1/2} }  ,  \quad 
   {\cal NLI}_2^k = - 2h^2\ciptwo{  D_y \left( {\cal NLE}^k \right) }
  {D_y \hat{e}^{k+1/2} } .  \nonumber 
	\end{eqnarray}

We focus on the first term ${\cal NLI}_1^k$; the second term 
${\cal NLI}_2^k$ can be analyzed in the same way. In $x$ direction, we 
drop the subscript $j$ in the grid index, just for simplicity of presentation. 
A detailed expansion 
shows that 
\begin{eqnarray} 
   &&
    D_x \left( \eta \left(  \phi^k , \phi^{k+1} \right)  \right)_{i+1/2}   
   \nonumber 
\\
  &=&
   \frac{1}{4 h} \left(   \left(  ( \phi_{i+1}^{k+1} )^2 
  + ( \phi_{i+1}^k )^2  \right) 
   \left(  \phi_{i+1}^{k+1} + \phi_{i+1}^k  \right) 
   - \left(  ( \phi_{i}^{k+1} )^2 + ( \phi_{i}^k )^2  \right) 
   \left(  \phi_{i}^{k+1} + \phi_{i}^k  \right)   \right)    \nonumber 
\\
  &=&
   \frac{1}{4 h}   \left(  ( \phi_{i+1}^{k+1} )^2 
  + ( \phi_{i+1}^k )^2  
  +  (  \phi_{i+1}^{k+1} +  \phi_{i}^{k+1} )  
      ( \phi_{i}^{k+1} + \phi_{i}^k ) \right) 
   \left(  \phi_{i+1}^{k+1} - \phi_{i}^{k+1} 
   + \phi_{i+1}^k  -  \phi_{i}^k   \right)  \nonumber 
\\
  &&
   + \frac{1}{4h} \left(  \phi_{i+1}^k - \phi_{i+1}^{k+1} 
  + \phi_{i}^k - \phi_{i}^{k+1} \right) 
   \left(  \phi_{i}^{k+1} + \phi_{i}^k  \right)  
   \left( \phi_{i+1}^k  -  \phi_{i}^k   \right)   \nonumber 
\\
  &=&
   \frac12   \left(  ( \phi_{i+1}^{k+1} )^2 
  + ( \phi_{i+1}^k )^2  
  +  (  \phi_{i+1}^{k+1} +  \phi_{i}^{k+1} )  
      ( \phi_{i}^{k+1} + \phi_{i}^k ) \right) 
   D_x \phi^{k+1/2}_{i+1/2}   \nonumber 
\\
  &&
   + \frac14 \left(  \phi_{i+1}^k - \phi_{i+1}^{k+1} 
  + \phi_{i}^k - \phi_{i}^{k+1} \right) 
   \left(  \phi_{i}^{k+1} + \phi_{i}^k  \right)  
   D_x \phi^k_{i+1/2}  .    \label{est-nonlinear-3}  
\end{eqnarray}
A similar expansion can be made for 
$ D_x \left( \eta \left(  \hat{\Phi}^k , \hat{\Phi}^{k+1} \right)  \right)_{i+1/2}$. 
In turn, we arrive at 
\begin{eqnarray} 
   &&
    D_x \left(  
     \eta \left(  \hat{\Phi}^k , \hat{\Phi}^{k+1} \right)  
  -  \eta \left(  \phi^k , \phi^{k+1} \right)  \right)_{i+1/2}    
  = {\cal NLE}^k_{1,1} + {\cal NLE}^k_{1,2} 
   +  {\cal NLE}^k_{1,3} + {\cal NLE}^k_{1,4}  ,   
   \nonumber 
\\
  &&
   {\cal NLE}^k_{1,1} =  \frac12 \left(  ( \phi_{i+1}^{k+1} )^2 
  + ( \phi_{i+1}^k )^2  
  +  (  \phi_{i+1}^{k+1} +  \phi_{i}^{k+1} )  
      ( \phi_{i}^{k+1} + \phi_{i}^k ) \right)  
   D_x \hat{e}^{k+1/2}_{i+1/2}   ,   \label{est-nonlinear-4-1} 
\\
  && 
   {\cal NLE}^k_{1,2} =  \frac12 \Biggl( \left( \hat{\Phi}_{i+1}^{k+1} 
   + \phi_{i+1}^{k+1} \right)   \hat{e}_{i+1}^{k+1}    
  +   \left( \hat{\Phi}_{i}^{k+1} + \phi_{i}^{k+1} \right)   
    \hat{e}_{i}^{k+1}    \nonumber 
\\
  && \qquad  \qquad 
  +    (  \hat{e}_{i+1}^{k+1} +  \hat{e}_{i}^{k+1} )  
      ( \phi_{i}^{k+1} + \phi_{i}^k ) 
   +  (  \hat{\Phi}_{i+1}^{k+1} +  \hat{\Phi}_{i}^{k+1} )  
      ( \hat{e}_{i}^{k+1} + \hat{e}_{i}^k )     \Biggr) 
   D_x \hat{\Phi}^{k+1/2}_{i+1/2}  ,    \label{est-nonlinear-4-2} 
\\
  &&
   {\cal NLE}^k_{1,3} =
    \frac14 \left(  \hat{\Phi}_{i+1}^k - \hat{\Phi}_{i+1}^{k+1} 
  + \hat{\Phi}_{i}^k - \hat{\Phi}_{i}^{k+1} \right) 
   \left(  \hat{\Phi}_{i}^{k+1} + \hat{\Phi}_{i}^k  \right)  
   D_x \hat{e}^k_{i+1/2}  ,    \label{est-nonlinear-4-3}  
\\
  &&
   {\cal NLE}^k_{1,4} =
    \frac14 \Biggl( \left(  \hat{e}_{i+1}^k - \hat{e}_{i+1}^{k+1} 
  + \hat{e}_{i}^k - \hat{e}_{i}^{k+1} \right) 
   \left(  \hat{\Phi}_{i}^{k+1} + \hat{\Phi}_{i}^k  \right)    \nonumber 
\\
  &&  \qquad  \qquad 
   +  \left(  \phi_{i+1}^k - \phi_{i+1}^{k+1} 
  + \phi_{i}^k - \phi_{i}^{k+1} \right) 
   \left(  \hat{e}_{i}^{k+1} + \hat{e}_{i}^k  \right)   \Biggr)
   D_x \phi^k_{i+1/2}  .    \label{est-nonlinear-4-4}  
\end{eqnarray}

For the term ${\cal NLE}^k_{1,1}$, we observe the following estimate for 
the nonlinear coefficient: 
\begin{eqnarray} 
  {\cal C}_1 &=&  ( \phi_{i+1}^{k+1} )^2 + ( \phi_{i+1}^k )^2  
  +  (  \phi_{i+1}^{k+1} +  \phi_{i}^{k+1} )  
      ( \phi_{i}^{k+1} + \phi_{i}^k )   \nonumber 
\\
   &=&  
   ( \phi_{i+1}^{k+1} )^2 + ( \phi_{i+1}^k )^2  
  +  \phi_{i+1}^{k+1} \phi_{i}^{k+1} +  ( \phi_{i}^{k+1} )^2  
   +   \phi_{i+1}^{k+1} \phi_{i}^k + \phi_{i}^{k+1} \phi_{i}^k   \nonumber
\\
   &\ge&  
   \frac12 \left( ( \phi_{i+1}^{k+1} )^2 + ( \phi_{i}^{k+1} )^2 \right) 
  + ( \phi_{i+1}^k )^2    
   +   \phi_{i+1}^{k+1} \phi_{i}^k + \phi_{i}^{k+1} \phi_{i}^k   \nonumber
\\
   &\ge&  
   ( \phi_{i+1}^k )^2  - ( \phi_{i}^k )^2 ,  \label{est-nonlinear-5}  
\end{eqnarray}
with a repeated application of Cauchy inequality in the last two steps. 
Meanwhile, the a-priori assumption (\ref{a priori-4}) for the numerical 
solution $\phi$ at time step $t^k$ indicates that 
\begin{eqnarray} 
  - {\cal C}_1 \le \left| \phi_{i+1}^k  + \phi_{i}^k  \right| 
  \cdot \left| \phi_{i+1}^k  - \phi_{i}^k  \right| 
 =  h  \left| \phi_{i+1}^k  + \phi_{i}^k  \right| 
  \cdot \left| D_x \phi_{i+1/2}^k  \right| 
  \le 2 h  \nrm{ \phi^k }_\infty \cdot \nrm{ \nabla_h \phi^k }_\infty 
  \le 2 C_0^2 h , 
  \label{est-nonlinear-6} 
\end{eqnarray} 
at a pointwise level. As a result, its combination with 
(\ref{est-nonlinear-4-1}) implies that 
\begin{eqnarray} 
   - 2h^2\ciptwo{ {\cal NLE}_{1,1}^k }{D_x \hat{e}^{k+1/2} } 
  \le 2 C_0^2 h  \nrm{ D_x \hat{e}^{k+1/2} }_2^2  . 
  \label{est-nonlinear-7} 
\end{eqnarray}

Similar estimates can be derived for ${\cal NLE}^k_{1,3}$. 
The regularity assumption (\ref{a priori-1}), (\ref{a priori-2}) for the 
constructed approximate solution $\hat{\Phi}$ shows that 
\begin{eqnarray} 
   \left| \hat{\Phi}_{i+1}^k - \hat{\Phi}_{i+1}^{k+1}  \right|  
   + \left| \hat{\Phi}_{i}^k - \hat{\Phi}_{i}^{k+1} \right| 
  \le 2 s \nrm{\frac{\hat{\Phi}^{k+1} - \hat{\Phi}^k }{s} }_{\infty}  
  \le  2 C_{0}  s ,   \label{est-nonlinear-8-1} 
\\
   \left| \hat{\Phi}_{i}^{k+1} + \hat{\Phi}_{i}^k  \right| 
  \le  \nrm{ \hat{\Phi}^{k+1} }_\infty + \nrm{ \hat{\Phi}^k }_{\infty}    
  \le 2 C_0 ,  \quad \mbox{so that}    \label{est-nonlinear-8-2}  
\\
  \left| {\cal C}_3 \right| 
  =  \left| \left(  \hat{\Phi}_{i+1}^k - \hat{\Phi}_{i+1}^{k+1} 
  + \hat{\Phi}_{i}^k - \hat{\Phi}_{i}^{k+1} \right) 
   \left(  \hat{\Phi}_{i}^{k+1} + \hat{\Phi}_{i}^k  \right)  \right|  
  \le 4 C_0^2 s  , \label{est-nonlinear-8-3} 
\end{eqnarray}
at a point-wise level. In turn, we arrive at 
\begin{eqnarray} 
   - 2h^2\ciptwo{ {\cal NLE}_{1,3}^k }{D_x \hat{e}^{k+1/2} } 
  &\le& 2 C_0^2 s  \nrm{ D_x \hat{e}^{k+1/2} }_2 
  \cdot \nrm{ D_x \hat{e}^k }_2    \nonumber 
\\
  &\le&
      \frac18 \alpha  \nrm{ D_x \hat{e}^{k+1/2} }_2^2  
  +  \frac{8 C_0^4}{\alpha} s^2  \nrm{ D_x \hat{e}^k }_2^2 . 
  \label{est-nonlinear-9} 
\end{eqnarray}

For the second nonlinear term ${\cal NLE}^k_{1,2}$, we start from a 
rewritten form: 
\begin{eqnarray} 
   {\cal NLE}^k_{1,2}  &=&  
  \frac12  \Biggl( \left( \hat{\Phi}_{i+1}^{k+1} 
   + \phi_{i+1}^{k+1} + \phi_{i}^{k+1} + \phi_{i}^k \right)   
    \hat{e}_{i+1}^{k+1}    
  +   \left( \hat{\Phi}_{i}^{k+1} + 2 \phi_{i}^{k+1} 
    + \phi_{i}^k \right)   
    \hat{e}_{i}^{k+1}    \nonumber 
\\
  &&  
   +  2 (  \hat{\Phi}_{i+1}^{k+1} +  \hat{\Phi}_{i}^{k+1} )  
        \hat{e}_{i}^{k+1/2}     \Biggr)  
    D_x \hat{\Phi}^{k+1/2}_{i+1/2}  ,    \nonumber 
\\
  &=&  
     \Biggl( {\cal C}_{2,1}  \hat{e}_{i+1}^{k+1/2}    
  +   {\cal C}_{2,2}  \hat{e}_{i}^{k+1/2}     
  + \frac12 {\cal C}_{2,3}  \hat{e}_{i+1}^k  
  + \frac12 {\cal C}_{2,4}  \hat{e}_{i}^k \Biggr)  
   D_x \hat{\Phi}^{k+1/2}_{i+1/2}  ,   \label{est-nonlinear-10}   
\\
  \mbox{with} && 
   {\cal C}_{2,1} = - {\cal C}_{2,3} = \hat{\Phi}_{i+1}^{k+1} 
   + \phi_{i+1}^{k+1} + \phi_{i}^{k+1} + \phi_{i}^k ,  \, \, 
  {\cal C}_{2,4} = -  \left( \hat{\Phi}_{i}^{k+1} + 2 \phi_{i}^{k+1} 
    + \phi_{i}^k \right) ,  \nonumber 
\\
   && 
  {\cal C}_{2,2} = 2 \hat{\Phi}_{i}^{k+1} + \hat{\Phi}_{i+1}^{k+1} 
   + 2 \phi_{i}^{k+1} + \phi_{i}^k .  \nonumber 
\end{eqnarray} 
For these nonlinear coefficients, it is clear that 
\begin{eqnarray} 
  \nrm{ {\cal C}_{2,1} }_4 + \nrm{ {\cal C}_{2,2} }_4  
  \le C  \left(  \nrm{\hat{\Phi}^k}_4   
  +  \nrm{\hat{\Phi}^{k+1} }_4 + \nrm{\phi^k}_4 
  +  \nrm{\phi^{k+1} }_4  \right)  \le C C_{0}  ,   
   \label{est-nonlinear-11-1} 
\\
   \nrm{ {\cal C}_{2,3} }_\infty + \nrm{ {\cal C}_{2,4} }_\infty  
  \le C  \left(  \nrm{\hat{\Phi}^k}_\infty   
  +  \nrm{\hat{\Phi}^{k+1} }_\infty + \nrm{\phi^k}_\infty 
  +  \nrm{\phi^{k+1} }_\infty  \right)  
   \le C \left( C_{0}   +  \nrm{\phi^{k+1} }_\infty  \right) ,   
   \label{est-nonlinear-11-2} 
\end{eqnarray}
in which the regularity condition (\ref{a priori-1}) and 
a-priori assumption (\ref{a priori-3})-(\ref{a priori-4}) were repeated 
used in the derivation. In particular, we note that the $\nrm{ \cdot }_4$ bound 
is available for both the approximate solution $\hat{\Phi}$ and the numerical solution $\phi$, at both time steps $t^k$ and $t^{k+1}$, and the same for 
the $\nrm{ \cdot }_\infty$ bound for $\hat{\Phi}$. Meanwhile, in 
$\nrm{ \cdot }_\infty$ norm for the numerical solution $\phi$, we only have 
its bound at time step $t^k$, as an a-priori assumption, and its bound at the 
next time step $t^{k+1}$ has to be obtained by a higher order convergence 
in $\ell^2$ norm via an inverse inequality, as will be shown later. 
As a result, an application of discrete H\"older 
inequality shows that 
\begin{eqnarray} 
  &&
   - 2h^2\ciptwo{ {\cal NLE}_{1,2}^k }{D_x \hat{e}^{k+1/2} }  \nonumber 
\\
  &\le& 
   \Bigl(  2  \left(  \nrm{ {\cal C}_{2,1} }_4 
   + \nrm{ {\cal C}_{2,2} }_4  \right)  \nrm{ \hat{e}^{k+1/2} }_4 
  +   \left( \nrm{ {\cal C}_{2,3} }_\infty 
   + \nrm{ {\cal C}_{2,4} }_\infty  \right)  \nrm{ \hat{e}^k }_2  \Bigr) 
   \nrm{ D_x \hat{e}^{k+1/2} }_2    \nonumber 
\\
  &\le&
   \Bigl(  C C_0  \nrm{ \hat{e}^{k+1/2} }_4 
  +    C \left( C_{0}   +  \nrm{\phi^{k+1} }_\infty  \right)  
   \nrm{ \hat{e}^k }_2  \Bigr)  \nrm{ D_x \hat{e}^{k+1/2} }_2   . 
  \label{est-nonlinear-12} 
\end{eqnarray}
Furthermore, a discrete Sobolev embedding in 2-D gives 
\begin{eqnarray} 
   \nrm{ \hat{e}^{k+1/2} }_4  
  \le  C  \nrm{ \hat{e}^{k+1/2} }_2^{\frac12}   \cdot 
  \nrm{ \nabla_h \hat{e}^{k+1/2} }_2^{\frac12}  ,   \quad 
  \mbox{since}   
  \, \, \ciptwo{ \hat{e}^{k+1/2} }{{\bf 1}} = 0  , 
   \label{est-nonlinear-13} 
\end{eqnarray}
with $C$ independent on $h$; its proof can be found in \cite{guan12a}.  
We note that the zero-mean property of $\hat{e}^{k+1/2}$ comes from 
(\ref{consistency-12}). 
Therefore, the first part in (\ref{est-nonlinear-12}) can be bounded by 
\begin{eqnarray} 
    C C_0  \nrm{ \hat{e}^{k+1/2} }_4  \cdot \nrm{ D_x \hat{e}^{k+1/2} }_2  
   \le M \nrm{ \hat{e}^{k+1/2} }_2^{\frac12}   \cdot 
  \nrm{ \nabla_h \hat{e}^{k+1/2} }_2^{\frac32}  ,  \quad 
   \mbox{with}  \, \, \, M = C C_ 0 . 
\end{eqnarray}
In addition, we use the Young inequality
	\begin{equation} 
a \cdot b \le \frac{a^p}{p} + \frac{b^q}{q} ,  \quad  \forall \ a, b > 0 ,  \quad  \frac{1}{p} + \frac{1}{q} = 1 ,  
  \label{Young-1}
	\end{equation}
with the choices $p=4$, $q=\frac43$, $a = \left( 6 \alpha^{-1} \right)^{\frac34}  M \nrm{ \hat{e}^{k+1/2}  }_2^{\frac12}$,  $b = \left( \frac16 \alpha \right)^{\frac34}  \nrm{ \nabla_h \hat{e}^{k+1/2}  }_2^{\frac32}$, and get  
	\begin{eqnarray} 
  &&
   C C_0  \nrm{ \hat{e}^{k+1/2} }_4  \cdot \nrm{ D_x \hat{e}^{k+1/2} }_2 
  \le  M \nrm{ \hat{e}^{k+1/2} }_2^{\frac12}   \cdot 
  \nrm{ \nabla_h \hat{e}^{k+1/2} }_2^{\frac32}   \nonumber 
\\
  &=& a \cdot b \le  \frac14 a^4 + \frac34 b^{\frac43}
    = \frac14  M^4 \cdot \frac{\left( 6 \right)^3}{\alpha^3}  
  \nrm{ \hat{e}^{k+1/2} }_2^2 
  + \frac{\alpha}{8} \nrm{ \nabla_h \hat{e}^{k+1/2} }_2^2 . 
	\label{est-nonlinear-14} 
	\end{eqnarray}
The bound for the second part in (\ref{est-nonlinear-12}) can be obtained 
by Cauchy inequality: 
\begin{eqnarray} 
    C \left( C_{0}   +  \nrm{\phi^{k+1} }_\infty  \right)  
   \nrm{ \hat{e}^k }_2  \nrm{ D_x \hat{e}^{k+1/2} }_2 
    \le  \frac{C \left( C_0^2 
  + \nrm{\phi^{k+1} }_\infty^2 \right) }{\alpha}     
    \nrm{ \hat{e}^k }_2^2  
   +  \frac18 \alpha \nrm{ D_x \hat{e}^{k+1/2} }_2^2 . 
    \label{est-nonlinear-15}  
\end{eqnarray}
Consequently, a combination of (\ref{est-nonlinear-12}), 
(\ref{est-nonlinear-14}), (\ref{est-nonlinear-15}) yields 
\begin{eqnarray} 
   - 2h^2\ciptwo{ {\cal NLE}_{1,2}^k }{D_x \hat{e}^{k+1/2} }  
    &\le&  \frac{C_8}{\alpha^3}  
   \left(  \nrm{ \hat{e}^k }_2^2 + \nrm{ \hat{e}^{k+1} }_2^2  \right) 
  + \frac{C \nrm{\phi^{k+1} }_\infty^2 }{\alpha}  \nrm{ \hat{e}^k }_2^2   
   \nonumber 
\\
  && 
  + \frac{\alpha}{8} \left( \nrm{ \nabla_h \hat{e}^{k+1/2} }_2^2
  +  \nrm{ D_x \hat{e}^{k+1/2} }_2^2  \right)  . 
  \label{est-nonlinear-16} 
\end{eqnarray}
with $C_8 = 27 M^4 + C C_0^2$. 
 
The analysis for the fourth nonlinear term ${\cal NLE}^k_{1,4}$ is similar to 
that of ${\cal NLE}^k_{1,2}$. Its rewritten form reads: 
\begin{eqnarray} 
   {\cal NLE}^k_{1,4}  &=& 
   \frac14 \Biggl( \left(  \hat{e}_{i+1}^k - \hat{e}_{i+1}^{k+1} 
  + \hat{e}_{i}^k - \hat{e}_{i}^{k+1} \right) 
   \left(  \hat{\Phi}_{i}^{k+1} + \hat{\Phi}_{i}^k  \right)     \nonumber 
\\
  &&   
   +  2 \left(  \phi_{i+1}^k - \phi_{i+1}^{k+1} 
  + \phi_{i}^k - \phi_{i}^{k+1} \right)    \hat{e}_{i}^{k+1/2}   \Biggr)
   D_x \phi^k_{i+1/2}  ,    \nonumber 
\\
  &=&  
     \Biggl( {\cal C}_{4,1}  \hat{e}_{i+1}^{k+1/2}    
  +   {\cal C}_{4,2}  \hat{e}_{i}^{k+1/2}     
  + {\cal C}_{4,3}  \hat{e}_{i+1}^k  
  + {\cal C}_{4,4}  \hat{e}_{i}^k \Biggr)  
   D_x \hat{\Phi}^{k+1/2}_{i+1/2}  ,   \label{est-nonlinear-17}   
\\
  \mbox{with} && 
   {\cal C}_{4,1} = - \frac12 
   \left(  \hat{\Phi}_{i}^{k+1} + \hat{\Phi}_{i}^k  \right)  ,   \quad 
   {\cal C}_{4,2} = \frac12  \left(  \phi_{i+1}^k - \phi_{i+1}^{k+1} 
  + \phi_{i}^k - \phi_{i}^{k+1}  
  -  \hat{\Phi}_{i}^{k+1} - \hat{\Phi}_{i}^k  \right) , \nonumber 
\\
  &&
   {\cal C}_{4,3} = {\cal C}_{4,4} = \frac12  
   \left(  \hat{\Phi}_{i}^{k+1} + \hat{\Phi}_{i}^k  \right) . \nonumber 
\end{eqnarray} 
Similarly, these nonlinear coefficients can be bounded by 
\begin{eqnarray} 
    \nrm{ {\cal C}_{4,1} }_4 + \nrm{ {\cal C}_{4,2} }_4  
  \le C  \left(  \nrm{\hat{\Phi}^k}_4   
  +  \nrm{\hat{\Phi}^{k+1} }_4 + \nrm{\phi^k}_4 
  +  \nrm{\phi^{k+1} }_4  \right)  \le C C_{0}  ,   
   \label{est-nonlinear-18-1} 
\\
   \nrm{ {\cal C}_{4,3} }_\infty + \nrm{ {\cal C}_{4,4} }_\infty  
  \le C  \left(  \nrm{\hat{\Phi}^k}_\infty   
  +  \nrm{\hat{\Phi}^{k+1} }_\infty \right)  
   \le C C_{0} .      
   \label{est-nonlinear-18-2} 
\end{eqnarray}
Note that for ${\cal C}_{4,3}$ and ${\cal C}_{4,4}$, since the numerical solution $\phi$ is not involved, the regularity assumption 
(\ref{a priori-1}),  (\ref{a priori-2}) for the approximate solution $\hat{\Phi}$ 
directly gives a bounded for these two coefficients. This also greatly simplifies 
the analysis below. Then we have 
\begin{eqnarray} 
  &&
   - 2h^2\ciptwo{ {\cal NLE}_{1,4}^k }{D_x \hat{e}^{k+1/2} }  \nonumber 
\\
  &\le& 
   \Bigl(  2  \left(  \nrm{ {\cal C}_{4,1} }_4 
   + \nrm{ {\cal C}_{4,2} }_4  \right)  \nrm{ \hat{e}^{k+1/2} }_4 
  +   \left( \nrm{ {\cal C}_{4,3} }_\infty 
   + \nrm{ {\cal C}_{4,4} }_\infty  \right)  \nrm{ \hat{e}^k }_2  \Bigr) 
   \nrm{ D_x \hat{e}^{k+1/2} }_2    \nonumber 
\\
  &\le&
   C C_0  \left( \nrm{ \hat{e}^{k+1/2} }_4  
  +  \nrm{ \hat{e}^k }_2  \right)  \cdot \nrm{ D_x \hat{e}^{k+1/2} }_2   . 
  \label{est-nonlinear-19} 
\end{eqnarray}
In turn, the estimates (\ref{est-nonlinear-13})-(\ref{est-nonlinear-15}) 
are also valid; consequently, the following estimate can be derived: 
\begin{eqnarray} 
   - 2h^2\ciptwo{ {\cal NLE}_{1,4}^k }{D_x \hat{e}^{k+1/2} }  
    &\le&  \frac{C_9}{\alpha^3}  
   \left(  \nrm{ \hat{e}^k }_2^2 + \nrm{ \hat{e}^{k+1} }_2^2  \right)    
   \nonumber 
\\
  && 
  + \frac{\alpha}{8} \left( \nrm{ \nabla_h \hat{e}^{k+1/2} }_2^2
  +  \nrm{ D_x \hat{e}^{k+1/2} }_2^2  \right)  . 
  \label{est-nonlinear-20} 
\end{eqnarray}

Finally, a combination of (\ref{est-nonlinear-7}), (\ref{est-nonlinear-9}), 
(\ref{est-nonlinear-16}) and (\ref{est-nonlinear-20}) reveals that 
\begin{eqnarray} 
   {\cal NLI}_1^k 
  &=& - 2h^2\ciptwo{  D_x \left( {\cal NLE}^k \right) }
  {D_x \hat{e}^{k+1/2} }    \nonumber 
\\
  &\le&
    \frac{C_{10}}{\alpha^3}  
   \left(  \nrm{ \hat{e}^k }_2^2 + \nrm{ \hat{e}^{k+1} }_2^2  \right)   
   + \frac{C \nrm{\phi^{k+1} }_\infty^2 }{\alpha}  \nrm{ \hat{e}^k }_2^2   
   \nonumber 
\\
  && 
  + \frac{\alpha}{4}  \nrm{ \nabla_h \hat{e}^{k+1/2} }_2^2 
  +  \frac{\alpha}{2} \nrm{ D_x \hat{e}^{k+1/2} }_2^2  
  +  \frac{8 C_0^4}{\alpha} s^2  \nrm{ D_x \hat{e}^k }_2^2  , 
  \label{est-nonlinear-21} 
\end{eqnarray}
by choosing $h$ with $2 C_0^2 h \le \frac{\alpha}{8}$. 
The analysis for ${\cal NLI}_2^k$ is essentially the same: 
\begin{eqnarray} 
   {\cal NLI}_2^k 
  &=& - 2h^2\ciptwo{  D_y \left( {\cal NLE}^k \right) }
  {D_y \hat{e}^{k+1/2} }    \nonumber 
\\
  &\le&
    \frac{C_{10}}{\alpha^3}  
   \left(  \nrm{ \hat{e}^k }_2^2 + \nrm{ \hat{e}^{k+1} }_2^2  \right)   
   + \frac{C \nrm{\phi^{k+1} }_\infty^2 }{\alpha}  \nrm{ \hat{e}^k }_2^2   
   \nonumber 
\\
  && 
  + \frac{\alpha}{4}  \nrm{ \nabla_h \hat{e}^{k+1/2} }_2^2 
  +  \frac{\alpha}{2} \nrm{ D_y \hat{e}^{k+1/2} }_2^2  
  +  \frac{8 C_0^4}{\alpha} s^2  \nrm{ D_y \hat{e}^k }_2^2  ; 
  \label{est-nonlinear-22} 
\end{eqnarray}
and the details are skipped for brevity of presentation. Therefore, 
a combination of (\ref{est-nonlinear-21}) and (\ref{est-nonlinear-22}) 
results in (\ref{est-nonlinear-1}). The proof of Proposition~\ref{estimate-nonlinear} 
is complete. 
\end{proof}

\begin{rem}
In fact, for the nonlinear error term, the form of expansion and decomposition 
in its discrete gradient is not unique. However, the way in our decomposition 
(\ref{est-nonlinear-4-1})-(\ref{est-nonlinear-4-4}) greatly facilitates the 
convergence analysis. 

It is well known that the exact solution $\Phi$ and the nonlinear potential $\Phi^3$ 
have a non-positive $H^{-1}$ inner product, since $3 \Phi^2 \ge 0$. 
However, for the second order numerical approximation 
$\eta \left( \phi^k , \phi^{k+1} \right)$, introduced by (\ref{nlch-scheme-3}), 
its error estimate becomes much more tricky. In the decomposition 
(\ref{est-nonlinear-4-1}), the nonlinear coefficient ${\cal C}_1$ is proven 
to be ``almost" non-negative, as in (\ref{est-nonlinear-5}), and the remainder 
term has an $O (h)$ bound given by (\ref{est-nonlinear-6}), using 
the $W^{1,\infty}$ bound assumption for the numerical solution at $t^k$, 
as given by (\ref{a priori-4}).  This treatment 
assures a controlled property of the nonlinear inner product associated with 
(\ref{est-nonlinear-4-1}). 

Moreover, since the numerical solution $\phi$ is involved with the nonlinear 
coefficient ${\cal C}_1$ in (\ref{est-nonlinear-4-1}), we could take the 
discrete gradient of the approximate solution $\Phi$ in the nonlinear 
expansion (\ref{est-nonlinear-4-2}), and its $\nrm{ \cdot }_\infty$ norm 
is directly bounded by (\ref{a priori-1}). If it is replaced by the discrete 
gradient of the numerical solution, a numerical analysis is not feasible, 
since a bound for $\nrm{ \phi^{k+1} }_\infty$ is not avialable at time 
step $t^{k+1}$. 

Meanwhile, in the nonlinear expansion (\ref{est-nonlinear-4-4}), 
an appearance of the discrete gradient of the numerical solution 
at time step $t^k$ does not cause any theoretical trouble, since 
we have had an a-priori bound (\ref{a priori-4}), which is to be 
recovered by an $O (s^3 + h^4)$ convergence analysis in $\ell^2$ norm. 

For the nonlinear errors appearing in (\ref{est-nonlinear-4-2}), 
(\ref{est-nonlinear-4-4}), we have to rewrite them in terms of 
a nonlinear combination of $\hat{e}^{k+1/2}$ and $\hat{e}^k$. 
The reason is that we only have a well-posed diffusion term of 
$\nrm{ \nabla_h \hat{e}^{k+1/2} }_2^2$; a positive diffusion term 
in either the form of $\nrm{ \nabla_h \hat{e}^{k} }_2^2$ or 
$\nrm{ \nabla_h \hat{e}^{k+1} }_2^2$ is not available in the 
numerical analysis, because of the second order 
numerical approximation. With such a rewriting, the terms involving 
$\hat{e}^{k+1/2}$ only require an $\ell^4$ bound for the numerical 
and approximate solutions, given by (\ref{est-nonlinear-11-1}), 
and the $\ell^4$ estimate for $\hat{e}^{k+1/2}$ is obtained by 
(\ref{est-nonlinear-13}), a discrete Sobolev embedding. In turn, 
these terms can be controlled with the help of Young inequality, 
as in (\ref{est-nonlinear-14}). 

  The terms involving $\hat{e}^k$ can be 
handled by a standard Cauchy inequality, and a coefficient 
$\nrm{ \phi^{k+1} }_\infty^2$ has to be included in the estimate 
(\ref{est-nonlinear-15}). Such a bound is not available at present; it has 
to be obtained from a preliminary estimate before a discrete Gronwall 
inequality is applied; see the analysis in later subsections.  

For the nonlinear expansion in (\ref{est-nonlinear-4-3}), 
we make the nonlinear coefficient of order $O (s)$, as analyzed 
by (\ref{est-nonlinear-8-1})-(\ref{est-nonlinear-8-3}). In addition, such 
a nonlinear coefficient has to be $\hat{\Phi}$ dependent, instead of 
$\phi$ dependent, since we have not had the divided difference bound 
(in time) for the numerical solution. With such an $O (s)$ analysis, 
the nonlinear inner product associated with (\ref{est-nonlinear-4-3}) 
is bounded by (\ref{est-nonlinear-9}), in which the first part can be 
controlled by the diffusion term, and the second part is an $O (s^2)$ 
increment. The stability of such an $O (s^2)$ incremental term is ensured 
by the term $\nrm{ \nabla_h \hat{e}^{k+1} }_2^2   
  - \nrm{ \nabla_h \hat{e}^{k} }_2^2$, which appears in (\ref{est-linear-2}) 
in Proposition~\ref{estimate-linear}, the estimate of the concave diffusion term. 
\end{rem}

	\begin{rem}
	\label{rem-3D-1}
For the 3D case, a discrete Sobolev embedding gives 
\begin{eqnarray} 
   \nrm{ \hat{e}^{k+1/2} }_4  
  \le  C  \nrm{ \hat{e}^{k+1/2} }_2^{\frac14}   \cdot 
  \nrm{ \nabla_h \hat{e}^{k+1/2} }_2^{\frac34}  ,   \quad 
  \mbox{if}   
  \, \, \ciptwo{ \hat{e}^{k+1/2} }{{\bf 1}} = 0  , 
   \label{embedding-l4-3D} 
\end{eqnarray}
which is analogous to (\ref{est-nonlinear-13}) in 2-D; also see the related 
discussions in \cite{guan12a}. In turn, we are able to derive the following result
	\begin{eqnarray}
  &&
2h^3 \cipthree{  \eta \left(  \hat{\Phi}^k , \hat{\Phi}^{k+1} \right) 
  - \eta \left(  \phi^k , \phi^{k+1} \right)  }{\Delta_h 
  \hat{e}^{k+1/2} }   \nonumber 
\\
  &\le&
      \frac{C_1}{\alpha^7} \left(  \nrm{ \hat{e}^k }_2^2  
   + \nrm{ \hat{e}^{k+1} }_2^2  \right)
   + \alpha \nrm{\nabla_h  \hat{e}^{k+1/2}  }_2^2   \nonumber 
\\
  &&  
   + \frac{C_1}{\alpha} s^2 \nrm{\nabla_h  \hat{e}^k  }_2^2
   + \frac{C_1}{\alpha} \left( 1 + \nrm{ \phi^{k+1} }_\infty^2 \right)  
    \nrm{ \hat{e}^k }_2^2 ,  \quad  \forall \alpha > 0 ,   
  \label{est-nonlinear-3D}
	\end{eqnarray}
the only changes being the $\alpha^7$ replaces $\alpha^3$ and we use the triple summation $\cipthree{\, \cdot\, }{\, \cdot\, }$. As a result, a full order convergence in 3D  can be derived in the same manner. The details are omitted in this paper for the sake of brevity. 
	\end{rem}

\subsection{Proof of Theorem~\ref{thm:L2-converge}: 
$\ell^\infty (0,T; \ell^2)$ convergence}  

We begin with an $O (s^3 + h^4)$ convergence assumption of the numerical solution, in $\ell^2$ norm, up to time step $t^k$: 
\begin{eqnarray} 
  \nrm{ \hat{e}^j }_2  \le  C_{11}  e^{C_{12} t^j} \left( s^3 + h^4 \right) ,  
  \quad  \forall 0 \le j \le k  ,  \label{a priori-5}
\end{eqnarray} 
with $C_{11}$, $C_{12}$ independent on $s$ and $h$. Consequently, 
an application of inverse inequality shows that 
\begin{eqnarray} 
   \nrm{ \hat{e}^j }_\infty + \nrm{ \nabla_h \hat{e}^j }_\infty 
   \le  \frac{ C \nrm{ \hat{e}^j }_2 }{ h^{\frac{d}{2} +1} }
  \le  \frac{ C C_{11}  e^{C_{12} t^j} \left( s^3 + h^4 \right) }{ h^{\frac{d}{2} +1} } 
  \le  C h^{1/2}  \le 1 ,  \quad  \forall 0 \le j \le k  , \label{a priori-6} 
\end{eqnarray} 
with the dimension $d=2$ or 3. It is also noted that the linear refinement 
constraint, $s \le Ch$, is used in the above derivation. 
In turn, the a-priori assumption (\ref{a priori-4}) for the numerical solution 
at $t^k$ is valid by setting 
\begin{eqnarray} 
    C_0 = \max_{0 \le j \le k} 
   \left( \nrm{\hat{\Phi}^j}_{\infty} 
   + \nrm{\nabla_h \hat{\Phi}^j }_{\infty}  \right) + 1 .  
  \label{a priori-7} 
\end{eqnarray} 
Moreover, it is clear that an estimate for $\nrm{ \phi^{k+1} }_\infty$ 
is needed in the application of Proposition~\ref{estimate-nonlinear} in the 
nonlinear analysis. For this quantity, we observe that (\ref{a priori-3}), 
which comes from a global in time $\ell^4$ bound for the numerical solution 
(as derived in Theorem~\ref{lp-stable-thm}), implies that 
\begin{eqnarray} 
   \nrm{ \phi^{k+1} }_\infty  
   \le  \frac{ C \nrm{ \phi^{k+1} }_4 }{ h^{\frac{d}{4} } }
  \le  C C_0  h^{-\frac{d}{4} }   ,  \quad 
  \mbox{with $d$ the dimension}  , \label{a priori-8} 
\end{eqnarray} 
in which the first step comes from a similar inverse inequality. 

Now we derive the $\ell^2$ convergence at time step $t^{k+1}$. 
Multiplying by $2h^2 \hat{e}^{k+1/2} = h^2 ( \hat{e}^{k+1} + \hat{e}^k)$, summing over $i$ and $j$, and applying Green's second identity \eqref{green2ndthm-2d}, we have
	\begin{eqnarray}
   &&  
\nrm{\hat{e}^{k+1}}_2^2 - \nrm{ \hat{e}^k}_2^2  
  + 2 B_c s \nrm{ \nabla_h \hat{e}^{k+1/2} }_2^2  \nonumber 
\\
  &=& 
   2h^2s \ciptwo{  \eta \left(  \hat{\Phi}^k , \hat{\Phi}^{k+1} \right) 
  - \eta \left(  \phi^k , \phi^{k+1} \right)  }{\Delta_h 
  \hat{e}^{k+1/2} }   
    + 2sh^2 \ciptwo{\hat{\tau}^{k}}{ \hat{e}^{k+1/2}}  \nonumber 
	\\
&&- 2 B_e s h^2  
  \ciptwo{ \hat{\hat{e}}^{k+1/2}}{\Dh \hat{e}^{k+1/2}} 
  - 2 h^2s\ciptwo{\convone{J}{ \hat{\hat{e}}^{k+1/2}}}
  {\Dh \hat{e}^{k+1/2}}  .   \label{converge-3}
	\end{eqnarray}
Applying Propositions~\ref{estimate-linear}, ~\ref{estimate-nonlinear} 
for linear and nonlinear errors, and using the Cauchy inequality to 
bound the truncation error term: 
       \begin{eqnarray}
2h^2s\ciptwo{ \hat{\tau}^k }{\hat{e}^{k+1/2}} 
  \le s C_{13} 
  \left( s^3 + h^4 \right)^2 + s  \nrm{\hat{e}^{k+1/2}}_2^2 
  \le s C_{13}  \left( s^3 + h^4 \right)^2 
   + \frac{s}{2}   \left(  \nrm{\hat{e}^{k+1}}_2^2  
  +  \nrm{\hat{e}^k}_2^2  \right)  ,   
	\end{eqnarray}
we arrive at 
	\begin{eqnarray}
   &&  
\nrm{\hat{e}^{k+1}}_2^2 - \nrm{ \hat{e}^k}_2^2  
  + s  \left( 2 B_c - 5 B_e - 2 \alpha \right) 
   \nrm{ \nabla_h \hat{e}^{k+1/2} }_2^2   
  + B_e s \left( \nrm{ \nabla_h \hat{e}^{k+1} }_2^2   
  - \nrm{ \nabla_h \hat{e}^{k} }_2^2  \right)  \nonumber 
\\
  &\le& 
   \frac{C_1}{\alpha^3} s \left(  \nrm{ \hat{e}^k }_2^2  
   + \nrm{ \hat{e}^{k+1} }_2^2  \right)  
   + \frac{C_{14}}{\alpha} s \left(  \nrm{ \hat{e}^k }_2^2  
   + \nrm{ \hat{e}^{k-1} }_2^2  \right)    
   + B_e s \nrm{\nabla_h  \hat{e}^{k-1/2}  }_2^2  \nonumber 
\\
  && 
   + \frac{C_1}{\alpha} s^3 \nrm{\nabla_h  \hat{e}^k  }_2^2
   + \frac{C_1}{\alpha} s \nrm{ \phi^{k+1} }_\infty^2   
    \nrm{ \hat{e}^k }_2^2 
    +  s C_{13}  \left( s^3 + h^4 \right)^2 ,  \quad  \forall \alpha > 0 .   \label{converge-4}  
	\end{eqnarray}

\subsubsection{A preliminary estimate for $\nrm{ \phi^{k+1} }_\infty$ } 

Note that an $O (1)$ bound for $\nrm{ \phi^{k+1} }_\infty$ is not available 
at this point, due to the lack of information of the numerical solution at 
time step $t^{k+1}$. We only have (\ref{a priori-8}), which comes from an unconditional $\ell^4$ stability of the numerical solution, and this bound may become singular as $h \to 0$. Meanwhile, such a bound is needed to apply the 
Gronwall inequality. 

  To overcome this difficulty, we derive an estimate, based on 
(\ref{converge-4}), the assumption (\ref{a priori-5}) (up to time 
step $t^k$), and the preliminary bound (\ref{a priori-8}). 
The assumption (\ref{a priori-5}) implies that 
\begin{eqnarray} 
    \nrm{ \hat{e}^k }_2^2  \le C h^6 ,    \quad 
    \nrm{\nabla_h  \hat{e}^k  }_2^2 ,  
    \nrm{\nabla_h  \hat{e}^{k-1/2}  }_2^2  \le C h^4 ,  \quad 
    \nrm{ \phi^{k+1} }_\infty^2  \le  C  h^{-\frac{d}{2} }  , 
\end{eqnarray} 
with the standard constraint $s \le C h$. Furthermore, using the fact that 
$\gamma_0 = B_c - 3 B_e > 0$ and taking $\alpha = \frac{\gamma_0}{2}$, 
we conclude from (\ref{converge-4}) that 
     \begin{eqnarray}
     \nrm{\hat{e}^{k+1}}_2^2   
    \le C \left( h^5  + h^{7 - \frac{d}{2} } \right)  
   \le C h^5 ,   \quad 
  \mbox{since $7 - \frac{d}{2} > 5$ for $d=2,3$ } .  \label{a priori-9}  
	\end{eqnarray}
In turn, an application of inverse inequality shows that 
\begin{eqnarray} 
   \nrm{ \hat{e}^{k+1} }_\infty  
   \le  \frac{ C \nrm{ \hat{e}^{k+1} }_2 }{ h^{\frac{d}{2} } }
  \le  C  h^{\frac52-\frac{d}{2} }  \le C h \le 1 ,  \quad 
  \mbox{with $d=2$ or 3}  . \label{a priori-10} 
\end{eqnarray} 
Consequently, the triangular inequality yields 
\begin{eqnarray} 
   \nrm{ \phi^{k+1} }_\infty 
   \le \nrm{ \hat{\Phi}^{k+1} }_\infty +  \nrm{ \hat{e}^{k+1} }_\infty  
   \le C_{15} := C_0 + 1 .  \label{a priori-11} 
\end{eqnarray}

\begin{rem} 
Of course, the rough estimate (\ref{a priori-9}) is not the convergence result 
that we want. Not only its accuracy is not satisfactory, $O (h^{\frac52})$ 
instead of $O (s^3 + h^4)$, but also its stability is not maintained: 
$O (s^3 + h^4)$ convergence at the previous time step to an order 
$O (h^\frac52)$ at the next time step. The reason for such an 
accuracy deterioration is due to the singular bound (\ref{a priori-8}) 
for $\nrm{ \phi^{k+1} }_\infty$, which comes from the global in time 
$\ell^4$ bound for the numerical solution. The purpose of the rough 
estimate (\ref{a priori-9}) is to derive a preliminary ``convergence" 
result in the $\ell^2$ norm, based on the full convergence result at the 
previous time step, combined with the singular bound (\ref{a priori-8}), 
so that a regular $ O(1)$ bound can be obtained for the 
$\nrm{ \cdot }_\infty$ norm of the numerical solution at the next time 
step with an application of inverse inequality. Subsequently, the 
full order $\ell^2$ convergence at the next time step can be derived 
by using the discrete Gronwall inequality, since an $O (1)$ bound 
for $\nrm{ \phi^{k+1} }_\infty$ has been available. 
\end{rem}

\subsubsection{ $\ell^\infty (0,T; \ell^2 )$ convergence and 
a recovery of the assumption (\ref{a priori-5}) }

A substitution of (\ref{a priori-11}) into (\ref{converge-4}) gives
	\begin{eqnarray}
   &&  
\nrm{\hat{e}^{k+1}}_2^2 - \nrm{ \hat{e}^k}_2^2  
  + s  \left( 2 B_c - 5 B_e - 2 \alpha \right) 
   \nrm{ \nabla_h \hat{e}^{k+1/2} }_2^2   
  + B_e s \left( \nrm{ \nabla_h \hat{e}^{k+1} }_2^2   
  - \nrm{ \nabla_h \hat{e}^{k} }_2^2  \right)  \nonumber 
\\
  &\le& 
   \frac{C_{16}}{\alpha^3} s \left(  \nrm{ \hat{e}^k }_2^2  
   + \nrm{ \hat{e}^{k+1} }_2^2  \right)  
   + \frac{C_{14}}{\alpha} s \nrm{ \hat{e}^{k-1} }_2^2      
   + B_e s \nrm{\nabla_h  \hat{e}^{k-1/2}  }_2^2   \nonumber 
\\
  &&
   + \frac{C_1}{\alpha} s^3 \nrm{\nabla_h  \hat{e}^k  }_2^2 
   +  s C_{13}  \left( s^3 + h^4 \right)^2 ,  \quad  \forall \alpha > 0 .   
    \label{converge-5}  
	\end{eqnarray}
Replacing the index $k$ by $l$, summing on $l$, from $l=0$ to $l = k$, and using $\hat{e}^{0}\equiv0$ (by (\ref{consistency-12})), 
we have
	\begin{eqnarray}
&&
   \nrm{\hat{e}^{k+1} }_{2}^2  + B_e  s \nrm{ \nabla_h \hat{e}^{k+1} }_2^2 
  + s \left( 2 \gamma_0 - 2 \alpha \right) 
   \sum_{l=1}^{k} \nrm{\nabla_h \hat{e}^{l}}_{2}^2 
	\nonumber
	\\
&\le&   s \left( \frac{C_{16}}{\alpha^3} + \frac{C_{14} }{\alpha} \right) 
   \sum_{l=1}^{k}   \nrm{ \hat{e}^l }_2^2
   + s \frac{C_{16} }{\alpha^3} \sum_{l=0}^{k}  \nrm{ \hat{e}^{l+1} }_2^2 
   + \frac{C_1}{\alpha} s^3 \sum_{l=0}^{k} \nrm{\nabla_h  \hat{e}^l  }_2^2
  + s C_{13} \sum_{l=0}^{k}(h^4 + s^3)^2 ,    \label{converge-6}  
	\end{eqnarray}
with $B_c - 3 B_e = \gamma_0 >0$ as in \eqref{discrete-alpha-0}. 
As a direct consequence, by taking $\alpha = \frac{\gamma_0}{2}$, 
the following inequality holds:
	\begin{equation}
\frac{1}{1 - C_{17} s} \nrm{\hat{e}^{k+1}}_{2}^2 
   + B_e  s \nrm{ \nabla_h \hat{e}^{k+1} }_2^2 
  \le  s\frac{C_{18}}{\gamma_0^3} \sum_{l=1}^{k}\nrm{e^{l}}_{2}^2  
  + \frac{C_1}{\alpha} s^3 \sum_{l=0}^{k}  \nrm{\nabla_h  \hat{e}^l  }_2^2 
  +s C_{13} \sum_{l=0}^{k}(h^4 + s^3)^2  ,   \label{converge-7}  
	\end{equation}
with $C_{17} := \frac{C_{16}}{\alpha^3}$ and 
$C_{18} := C_{16} + C_{14}  \alpha^2 $.  We can always choose $s$ 
with $1 - C_{17} s \ge \frac12$. In turn, by denoting 
\begin{eqnarray} 
   G^l = 2 \nrm{\hat{e}^l}_{2}^2 
   + B_e  s \nrm{ \nabla_h \hat{e}^l }_2^2   ,   \label{energy-alternate-1} 
\end{eqnarray} 
we get 
       \begin{equation}
   G^{k+1}  \le  2 s\frac{C_{18}}{\gamma_0^3} \sum_{l=1}^{k}  G^l 
  +s C_{13} \sum_{l=0}^{k}(h^4 + s^3)^2  ,  \label{converge-8}  
	\end{equation}
with the choice of $s$ so that $\frac{C_1 s}{\alpha} \le 2 B_e$. 
An application of the discrete Gronwall inequality yields the desired result: 
\begin{eqnarray} 
    \nrm{\hat{e}^{k+1} }_{2}^2   \le G^{k+1}  
  \le  C_{19}  \left( s^3 + h^4 \right)^2 ,   
  \quad \mbox{so that}   \, \, \, 
     \nrm{\hat{e}^{k+1} }_{2}  \le  \sqrt{C_{19}}  \left( s^3 + h^4 \right)  ,  
   \label{converge-9} 
\end{eqnarray} 
with $C_{19}$ independent on $s$ and $h$. A more detailed exploration 
implies the structure of this constant: $C_{19} = C_{11} e^{C_{12} t^{k+1}}$. 
As a result, the a-priori assumption (\ref{a priori-5}) is recovered at 
the time step $t^{k+1}$ so that an $O (s^4 + h^4)$ convergence in 
$\ell^2$ norm, between the numerical solution and the constructed 
approximate solution $\hat{\Phi}$, has been established, using an induction 
argument. 

  Finally, the proof of Theorem~\ref{thm:L2-converge} can be completed with 
the following application of triangle inequality: 
\begin{eqnarray} 
    \nrm{ e^l }_{2}   =  \nrm{ \phi^l - \Phi^l }_{2} 
    \le  \nrm{ \phi^l - \hat{\Phi}^l }_{2}  
   +   \nrm{ \Phi^l - \hat{\Phi}^l }_{2}  
   \le  C  \left(  s^2 + h^2  \right)  ,    \quad 
  \forall  l \cdot s  \le T ,  
   \label{converge-10} 
\end{eqnarray} 
in which the error estimate (\ref{converge-9}) and the analysis 
(\ref{converge-1}) for the constructed solution are used. 

\begin{rem}  \label{rem:converge-nAC-3}  
  The assumption (\ref{discrete-alpha-0}) (for the physical parameters) is required in the convergence analysis for the nCH equation. Such an assumption is necessary for the convex diffusion part to control the concave diffusion part, due to a subtle estimate (\ref{est-linear-2}). As a consequence of this inequality, the assumption $B_c > 3 B_e$ has to be made to make the convergence analysis pass through. In comparison, for the nAC equation, this assumption is not required, as explained in Remark~\ref{rem:converge-nAC-2}.   
  
  On the other hand, our extensive numerical experiments have implied that, such an assumption only corresponds to a technical difficulty in the convergence analysis. For most practical computational models, the second order convergence is well-preserved as long as the 
positive-diffusivity condition (\ref{gamma-0}) is valid.  
\end{rem}

\begin{rem}   \label{rem:converge-nAC-4}
  We note that the second order $\ell^\infty (0,T; \ell^2)$ convergence for the nCH equation is conditional, i.e., under a mild linear refinement constraint, $s \le C h$. In comparison, the $\ell^\infty (0,T; \ell^2)$ convergence for the nAC equation is unconditional, as explained in Remark~\ref{rem:converge-nAC-1}. 
  
  Such a subtle difference comes from the analysis techniques for the nonlinear inner products. For the nAC equation, the decomposition (\ref{converge-nAC-6}) has greatly facilitated the error estimates, and the maximum norm bound of the numerical solution is not needed in the derivation. However, for the nCH equation, since the discrete $H^1$ inner product of $\hat{e}^{k+1/2}$ and the nonlinear error function has to be analyzed, we need to make an a-priori assumption (\ref{a priori-5}) at the previous time step, obtain a discrete $W^{1,\infty}$ bound of the numerical solution, and the $\ell^\infty (0,T; \ell^2)$ convergence estimate justifies the a-priori assumption at the next time step. This process is further facilitated by the higher order consistency analysis presented in Section~\ref{sec-consistency}. 
\end{rem}

\subsection{Proof of Theorem~\ref{thm:converge-infty}: 
$\ell^\infty (0,T; \ell^\infty)$ convergence}

With the $O (s^3 + h^4)$ convergence result (\ref{converge-9}), 
in $\ell^2$ norm, we apply the inverse inequality and get 
\begin{eqnarray} 
   \nrm{ \hat{e}^{k+1} }_\infty  
   \le  \frac{ C \nrm{ \hat{e}^{k+1} }_2 }{ h^{\frac{d}{2} } }
  \le  \frac{ C \sqrt{C_{19}}  \left( s^3 + h^4 \right)  }{ h^{\frac{d}{2} } } 
  \le C_{20} \left( s^2 + h^2 \right) ,  \quad 
  \mbox{with $d=2$}  , \label{converge-infty-2} 
\end{eqnarray} 
with the linear refinement constraint $s \le C h$ and 
$C_{20} = C \sqrt{C_{19}}$. For the 3-D case, a higher order asymptotic 
expansion of the numerical solution has to be performed so that 
an $O (s^4 + h^4)$ consistency and convergence in $\ell^2$ norm 
are obtained. The details are left to interested readers. 

  Subsequently, by combining the $\ell^\infty$ error estimate 
(\ref{converge-infty-2}) and the analysis 
(\ref{converge-1}) for the constructed solution, we finish the proof of Theorem~\ref{thm:L2-converge} with an application of triangle inequality: 
\begin{eqnarray} 
    \nrm{ e^l }_\infty   =  \nrm{ \phi^l - \Phi^l }_\infty 
    \le  \nrm{ \phi^l - \hat{\Phi}^l }_\infty  
   +   \nrm{ \Phi^l - \hat{\Phi}^l }_\infty  
   \le  C  \left(  s^2 + h^2  \right)  ,    \quad 
  \forall  l \cdot s  \le T .  
   \label{converge-infty-3} 
\end{eqnarray} 

\subsection{The $\ell^\infty (0,T; \ell^\infty)$ convergence for the nAC equation} 

For the second order convex splitting scheme (\ref{full-disc-cons-flow-nAC}) for the nAC equation, the higher order consistency analysis could be performed in the same manner as in Section~\ref{sec-consistency}. In turn, an $\ell^\infty (0,T; \ell^2)$ convergence estimate with an improved order $O (s^3 + h^4)$ is expected, and an application of inverse inequality leads to a similar $\ell^\infty (0,T; \ell^\infty)$ convergence result as Theorem~\ref{thm:converge-infty}, under the linear refinement path constraint $s \le C h$. The proof of the following theorem is skipped for brevity, and the details are left to interested readers.  

	\begin{thm}\label{thm:convergence-inf-nAC}
Under the assumptions of Theorem~\ref{thm:L2-converge-nAC}, we also have optimal order convergence of the numerical solution of the scheme (\ref{full-disc-cons-flow-nAC}) in the $\ell^\infty$ norm. Namely, if $s$ and $h$ are sufficiently small with the linear refinement path constraint $s \le C h$, with $C$ any fixed constant, we have 
	\begin{equation}
\nrm{e^l}_\infty \le C\left(h^2+s^2 \right), 
	\label{converge-infty-nAC-0}
	\end{equation} 
where $C>0$ is independent of $h$ and $s$.
	\end{thm}


	\section{Numerical results} \label{sec-numerical results}
In this section we present a few numerical experiments, verifying the convergence results of the second order schemes for the nCH and nAC equations. 
\subsection{Numerical convergence for the nCH equation}
Here we discuss the numerical results for the nCH equation. We present two cases, based on the restriction proposed in Eq.~\eqref{discrete-alpha-0}. These experiments verify the convergence rate in the $\ell^\infty(0,T;\ell^2)$ norm. We use a square domain $\Omega = (-0.5,0.5)^2$ with smooth, periodic initial data $0.5\sin(2\pi x_1)\cos(2 \pi x_2)$. The convolution kernel $\msfJ$ is taken to be 
	\begin{equation}\label{numerical_kernel}
\msfJ = \alpha \exp\left(-\frac{x_1^2+x_2^2}{\sigma^2}\right) , 
	\end{equation}
where $\sigma = 0.05$ and $\alpha = \frac{1}{\sigma^2}$. We extend $\msfJ$ periodically outside of $\Omega$. The other parameters are $\gamma_c =0$ and $\gamma_e = 1$ in the first case, which yields $\gamma_0 = \pi - 3 > 0$; and $\gamma_c =0$ and $\gamma_e = 2$ in the second case, which yields $\gamma_0 = \pi - 6 < 0$. The final time for the tests is given by $T=0.015625$. We take a linear refinement path $s = Ch$ with $C=0.1$, so that the global error is $O(h^2)$ in the $\ell^\infty(0,T;\ell^2)$ norm.  Since we do not have the exact solution -- these are not easily obtained for non-trivial convolution kernels -- we are using the difference between results on successive coarse and fine grids for the numerical comparison. The difference function, $e_{A}$, is evaluated at time $T=0.015625$ using the method described in~\cite{guan12a,hu09,wise10}.  The result is displayed in Tables~\ref{tab:error-1} and~\ref{tab:error-2}.  In both cases the global second-order accuracy of the method is confirmed when a linear refinement path is used.

	\begin{table}[htb]
	\centering	
	\begin{tabular}{|c|c|c|c|}
	\hline
coarse h &fine h& $\|e_{A}\|_2$ & rate
	\\
	\hline
1/128 & 1/256 & 0.003642747274850 & $-$
	\\
	\hline
1/256 & 1/512 & 0.000866930235764 & 2.071039102165065
	\\
	\hline
1/512 & 1/1024 &  0.000216260604309 & 2.003145023681790
	\\
	\hline
1/1024 & 1/2048 & 0.000054113342323 & 1.998714619254957
	\\
	\hline
	\end{tabular}
\caption{The difference between coarse and fine grids of the computed numerical solutions using a linear refinement path, with $\gamma_0 = \pi - 3 $.  The global second-order accuracy of the method is confirmed in the test.}
\label{tab:error-1}
	\end{table}

	\begin{table}[htb]
	\centering	
	\begin{tabular}{|c|c|c|c|}
	\hline
coarse h &fine h& $\|e_{A}\|_2$ & rate
	\\
	\hline
1/128 & 1/256 & 0.005355484518874 & $-$
	\\
	\hline
1/256 & 1/512 & 0.000483125827443 & 3.470546218670092
	\\
	\hline
1/512 & 1/1024 & 0.000139990250322 & 1.787072626339492
	\\
	\hline
1/1024 & 1/2048 & 0.000040099935792 & 1.803654521820036
	\\
	\hline
	\end{tabular}
\caption{The difference between coarse and fine grids of the computed numerical solutions using a linear refinement path, with $\gamma_0 = \pi - 6 $. The global second-order accuracy of the method is confirmed in the test.}
\label{tab:error-2}
	\end{table}

\subsection{Numerical convergence for the nAC equation}
Here we discuss the numerical results for the nAC equation. First we present the experiment verifying the numerical convergence rate. The setting of the experiment is the same as the nCH case, with $\gamma_c =0$ and $\gamma_e = 2$ which yields $\gamma_0 = \pi - 6 < 0$. The result is displayed in Tables~\ref{tab:error-3}.  The global second-order accuracy of the method is confirmed when a linear refinement path is used.

We also present experiments of phase separation described by the nAC equation under the following conditions: {\bf 1)} $\Omega = (-10,10)^2$; {\bf 2)} the size of time step is $s=0.01$, the number of nodes on grid is $512^2$ and the total number of time iterations is $10^4$; {\bf 3)} The convolution kernel $\msfJ$  is a function defined as the difference between two Gaussians:
\begin{equation}\label{kenerl:nuc-J}
 \msfJ = \alpha \exp\left(-\frac{x_1^2+x_2^2}{\sigma_1^2}\right) -\beta \exp\left(-\frac{x_1^2+x_2^2}{\sigma_2^2}\right) ,
\end{equation}
where $\sigma_1 = 0.16$, $\sigma_2 = 0.4$, $\alpha = \frac{0.1}{\sigma_1^2}$ and $\beta = \frac{0.08}{\sigma_2^2}$; {\bf 4)} $\gamma_e = 0$, $\gamma_c = 0.0$. The initial condition of the simulation is a random perturbation of the constant state $\phi_{ave}=0.0$. Figure~\ref{fig:evolve} shows snapshots of the evolution up to time T = 100, and Figure~\ref{fig:energy} shows the corresponding numerical energy for the simulation. The energy is observed to decay as time increases.

	\begin{table}[htb]
	\centering	
	\begin{tabular}{|c|c|c|c|}
	\hline
coarse h &fine h& $\|e_{A}\|_2$ & rate
	\\
	\hline
1/128 & 1/256 & 3.783500401280967e-05 & $-$
	\\
	\hline
1/256 & 1/512 & 9.458990514247017e-06 & 1.999963469064636
	\\
	\hline
1/512 & 1/1024 & 2.364753065037063e-06 & 1.999996683296409
	\\
	\hline
1/1024 & 1/2048 & 5.911879384151948e-07 & 2.000000800048265
	\\
	\hline
	\end{tabular}
\caption{The difference between coarse and fine grids of the computed numerical solutions using a linear refinement path, with $\gamma_0 = \pi - 6 $. The global second-order accuracy of the method is confirmed in the test.}
\label{tab:error-3}
	\end{table}

\begin{figure}
\centerline{\includegraphics[width=1.0\textwidth]{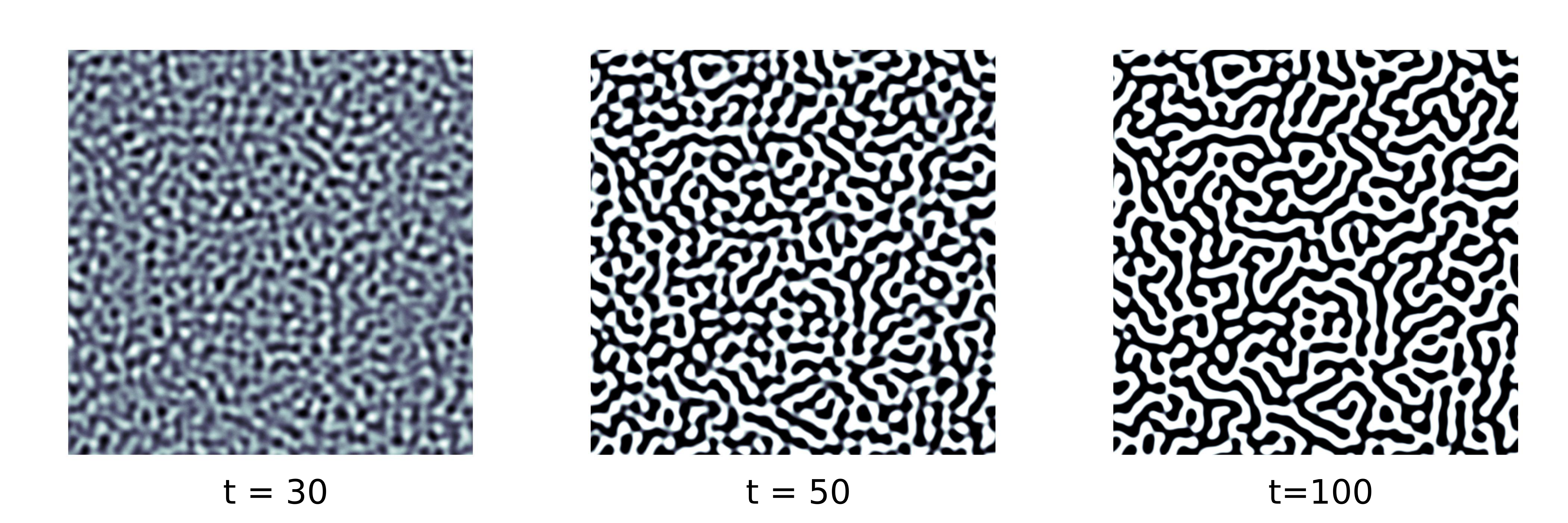}}
\caption{Phase separation described by nAC equation}
	\label{fig:evolve}
	\end{figure}	

\begin{figure}
\centerline{\includegraphics[width=0.50\textwidth]{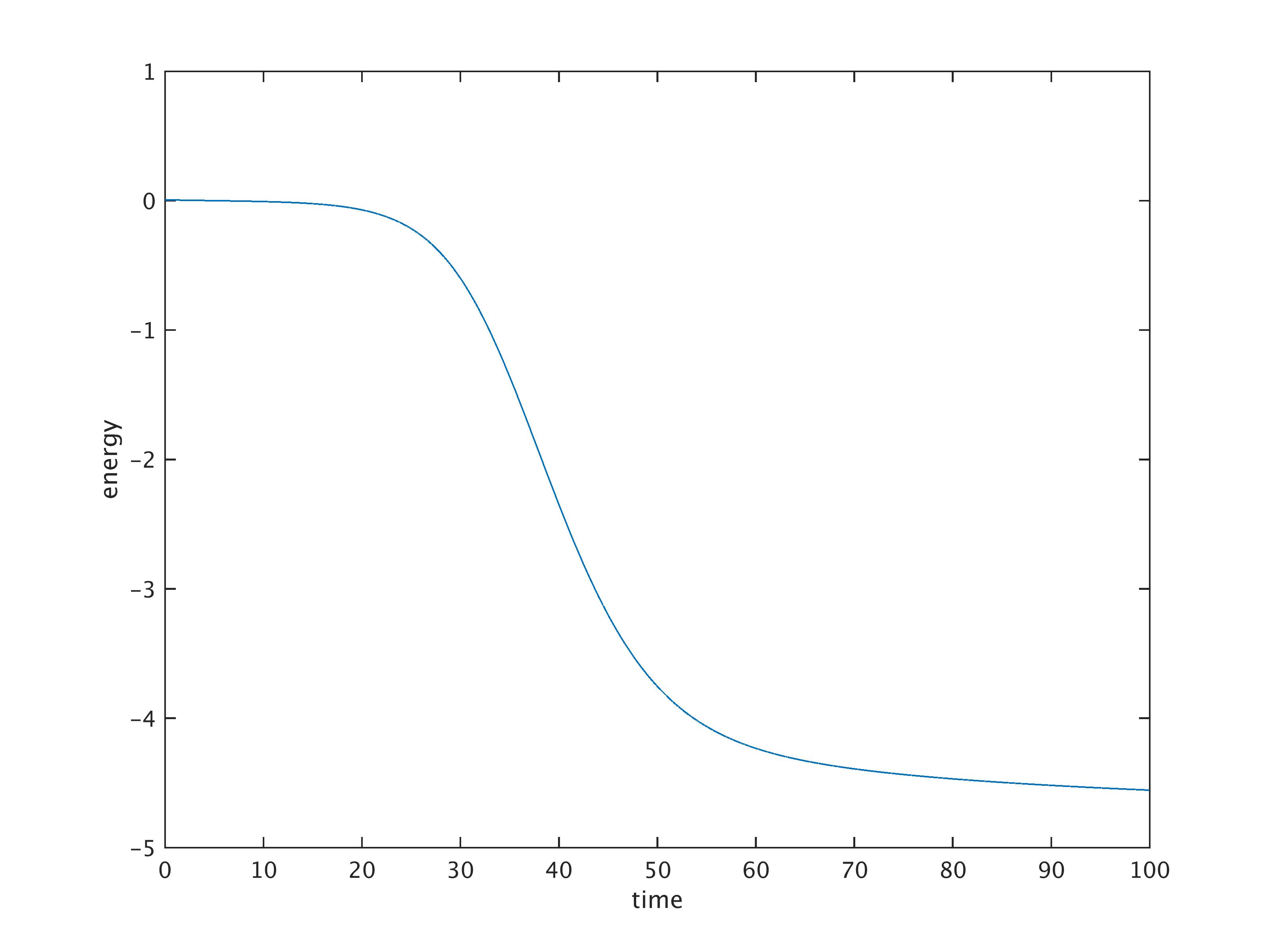}}
\caption{Energy evolution of phase separation}
	\label{fig:energy}
	\end{figure}
	
	\appendix

	\section{Finite difference discretization of space}
	\label{app-discrete-space}

Our primary goal in this appendix is to define some finite-difference operators and provide some summation-by-parts formulas in one and two space dimensions that are used to derive and analyze the numerical schemes. Everything extends straightforwardly to 3D.  We make extensive use of the notation and results for cell-centered functions from~\cite{wise09,wise10}.  The reader is directed to those references for more complete details.  

In 1D we will work on the interval $\Omega = (0,L)$, with $L = m\cdot h$, and in 2D, we work with the rectangle $\Omega = (0,L_1)\times(0,L_2)$, with $L_1 = m\cdot h$ and $L_2 = n\cdot h$, where $m$ and $n$ are positive integers and $h>0$ is the spatial step size.  Define $p_r := (r-\hf)\cdot h$, where $r$ takes on integer or half-integer values.   For any positive integer $\ell$, define $E_\ell = \left\{ p_r \ \middle|\ r=\frac{1}{2},\ldots, \ell+\frac{1}{2}\right\}$, $C_\ell = \left\{p_r \ \middle|\ r=1,\ldots, \ell\right\}$,  $C_{\overline{\ell}} = \left\{ p_r\cdot h\ \middle|\ r=0,\ldots, \ell+1\right\}$.   We need the 1D grid function spaces
	\begin{equation*}
{\mathcal C}_m = \left\{\phi: C_m\rightarrow \mathbb{R} \right\}, \ \  {\mathcal E}_m = \left\{u: E_m \rightarrow\mathbb{R} \right\},
	\end{equation*}
and the 2D grid function spaces
	\begin{align}
&{\mathcal C}_{m\times n} = \left\{\phi: C_m\times C_n \rightarrow \mathbb{R} \right\},  &{\mathcal V}_{m\times n} = \left\{f: E_m\times E_n \rightarrow\mathbb{R} \right\},
    	\nonumber
	\\
&{\mathcal E}^{\rm ew}_{m\times n} = \left\{u: E_m\times C_n \rightarrow\mathbb{R} \right\}, \ \  &{\mathcal E}^{\rm ns}_{m\times n} = \left\{v: C_m\times E_n \rightarrow\mathbb{R}  \right\} .
	\nonumber
	\end{align}

We use the notation $\phi_{i,j} := \phi\left(p_i,p_j\right)$ for \emph{cell-centered} functions, \emph{i.e.}, those in the space ${\mathcal C}_{m\times n}$. In component form \emph{east-west edge-centered} functions, \emph{i.e.}, those in the space ${\mathcal E}^{\rm ew}_{m\times n}$, are identified via $u_{i+\hf,j}:=u(p_{i+\hf},p_j)$.  In component form \emph{north-south edge-centered} functions, \emph{i.e.}, those in the space ${\mathcal E}^{\rm ns}_{m\times n}$,  are identified via $u_{i+\hf,j}:=u(p_{i+\hf},p_j)$. The functions of ${\mathcal V}_{m\times n}$ are called \emph{vertex-centered} functions.   In component form vertex-centered functions are identified via $f_{i+\hf,j+\hf}:=f(p_{i+\hf},p_{j+\hf})$.  The 1D cell-centered and edge-centered functions are easier to express. 

We will need the 1D grid inner-products $\cip{\, \cdot \,}{\, \cdot \,}$ and $\eip{\, \cdot \,}{\, \cdot \,}$ and the 2D grid inner-products $\ciptwo{\, \cdot \,}{\, \cdot \,}$, $\eipew{\, \cdot \,}{\, \cdot \,}$, $\eipns{\, \cdot \,}{\, \cdot \,}$ that are defined in~\cite{wise09,wise10}.

We shall say the cell-centered function $\phi\in {\mathcal C}_{m\times n}$ is periodic if and only if, for all $p,\, q\in\mathbb{Z}$,
	\begin{equation}
\phi_{i+ p\cdot m, j+ q\cdot n} = \phi_{i, j}  \quad i = 1,\ldots,m, \quad j = 1,\ldots, n .
	\end{equation}
Here we have abused notation a bit, since $\phi$ is not explicitly defined on an infinite grid.  Of course, $\phi$ can be extended as a periodic function in a perfectly natural way, which is the context in which we view the last definition. Similar definitions are implied for periodic edge-centered and vertex-centered grid functions.  The 1D and 3D cases are analogous and are suppressed.

The reader is referred to~\cite{wise09,wise10} for the precise definitions of the edge-to-center difference operators  $d_x : {\mathcal E}_{m\times n}^{\rm ew}\rightarrow{\mathcal C}_{m\times n}$ and $d_y : {\mathcal E}_{m\times n}^{\rm ns}\rightarrow{\mathcal C}_{m\times n}$; the $x-$dimension center-to-edge average and difference operators, respectively, $A_x,\, D_x: {\mathcal C}_{m\times n}\rightarrow{\mathcal E}_{m\times n}^{\rm ew}$;  the $y-$dimension center-to-edge average and difference operators, respectively, $A_y,\, D_y: {\mathcal C}_{m\times n}\rightarrow{\mathcal E}_{m\times n}^{\rm ns}$; and the standard 2D discrete Laplacian, $\Dh : {\mathcal C}_{m \times n}\rightarrow{\mathcal C}_{m\times n}$.  These operators have analogs in 1D and 3D that should be clear to the reader.

We will use the grid function norms defined in~\cite{wise09,wise10}.  The reader is referred to those references for the precise definitions of  $\nrm{\, \cdot \,}_2$, $\nrm{\, \cdot \,}_\infty$, $\nrm{\, \cdot \,}_p$ ($1\le p < \infty$), $\nrm{\, \cdot \,}_{0,2}$, $\nrm{\, \cdot \,}_{1,2}$, and $\nrm{\phi}_{2,2}$.  We will specifically use the following inverse inequality in 2D: for any $\phi\in{\mathcal C}_{m\times n}$ and all $1\le p < \infty$
	\begin{equation}
	\label{inverse-ineq}
\nrm{\phi}_{\infty}\le h^{-\frac{2}{p}}\nrm{\phi}_p .
	\end{equation}
Again, the analogous norms in 1D and 3D should be clear.

Using the definitions given in this appendix and in~\cite{wise09,wise10}, we obtain the following summation-by-parts formulas whose proofs are simple:
	\begin{prop}
	\label{sbp-2D-edge}
If $\phi\in{\mathcal C}_{m\times n}$ and $f\in{\mathcal E}_{m\times n}^{\rm ew}$ are periodic then
	\begin{equation}
h^2\, \eipew{D_x \phi}{f} = -h^2\, \ciptwo{\phi}{d_x f},
	\label{sbp-c-ew}
	\end{equation}
and if $\phi\in{\mathcal C}_{m\times n}$ and $f\in{\mathcal E}_{m\times n}^{\rm ns}$ are periodic then
	\begin{equation}
h^2\, \eipns{D_y\phi}{f} = -h^2\, \ciptwo{\phi}{d_y f} .
	\label{sbp-c-ns}
	\end{equation}
	\end{prop}
      
	\begin{prop}
	\label{green1stthm-2d}
Let $\phi,\, \psi\in {\mathcal C}_{\overline{m}\times\overline{n}}$ be periodic grid functions.  Then
	\begin{equation}
h^2\, \eipew{D_x\phi}{D_x\psi} +h^2\, \eipns{D_y\phi}{D_y\psi} = -h^2\, \ciptwo{\phi}{\Dh\psi} .
	\end{equation}
	\end{prop}

	\begin{prop}
	\label{green2ndthm-2d}
Let $\phi,\, \psi\in {\mathcal C}_{m \times n}$ be periodic grid functions.  Then
	\begin{equation}
h^2\, \ciptwo{\phi}{\Dh\psi} = h^2\, \ciptwo{\Dh\phi}{\psi} .
	\end{equation} 
	\end{prop}
Analogs in 1D and 3D of the summation-by-parts formulas above are straightforward. 

	\section*{Acknowledgements} 
The authors greatly appreciate many helpful discussions with Steven M. Wise, in particular for his insightful suggestion and comments.  JSL acknowledges partial support from NSF-CHE 1035218, NSF-DMR 1105409 and NSF-DMS 1217273. CW acknowledges partial support from NSF-DMS 1418689 and NSFC 11271281.

	\bibliographystyle{plain}
	\bibliography{NonLocalCH_2nd.bib}

	\end{document}